\documentclass[11pt]{article}

\usepackage[utf8]{inputenc} 
\usepackage[T1]{fontenc}    
\usepackage{hyperref}       
\usepackage{url}            
\usepackage{booktabs}       
\usepackage{amsfonts}       
\usepackage{nicefrac}       
\usepackage{microtype}      
\usepackage{amsmath}
\usepackage{tikz-cd}
\usepackage{subcaption}
\usepackage{zref}
\usepackage{algorithm,algpseudocode}
\DeclareMathOperator*{\argminA}{arg\,min}
\algnewcommand{\Inputs}[1]{%
  \State \textbf{Inputs:}
  \Statex \hspace*{\algorithmicindent}\parbox[t]{.8\linewidth}{\raggedright #1}
}
\algnewcommand{\Initialize}[1]{%
  \State \textbf{Initialize:}
  \Statex \hspace*{\algorithmicindent}\parbox[t]{.8\linewidth}{\raggedright #1}
}
\usepackage{amsthm,amsbsy,bm}
\usepackage{graphicx}
\graphicspath{{./}}
\newtheorem{definition}{Definition}

\newtheorem{theorem}{Theorem}[section]
\newtheorem{lemma}[theorem]{Lemma}
\newtheorem{corollary}[theorem]{Corollary}
\newtheorem{proposition}[theorem]{Proposition}
\title{Optimization and Learning With Nonlocal Calculus}

\author{%
  Sriram~Nagaraj\thanks{With the Federal Reserve Bank of Atlanta. Any views or opinions expressed in this article are those of the author alone and not necessarily those of the Federal Reserve Bank of Atlanta or the Federal Reserve System. This research did not receive any specific grant from funding agencies in the public, commercial, or not-for-profit sectors. Declarations of interest: none.}\\
  \texttt{sriram.nagaraj.atl@gmail.com} \\
}

\begin{document}

\maketitle

\begin{abstract}
Nonlocal models have recently had a major impact in nonlinear continuum mechanics and are used to describe physical systems/processes which cannot be accurately described by classical, calculus based ``local'' approaches. In part, this is due to their multiscale nature that enables aggregation of micro-level behavior to obtain a macro-level description of singular/irregular phenomena such as peridynamics, crack propagation, anomalous diffusion and transport phenomena. At the core of these models are \emph{nonlocal} differential operators, including nonlocal analogs of the gradient/Hessian. This paper initiates the use of such nonlocal operators in the context of optimization and learning. We define and analyze the convergence properties of nonlocal analogs of (stochastic) gradient descent and Newton's method on Euclidean spaces. Our results indicate that as the nonlocal interactions become less noticeable, the optima corresponding to nonlocal optimization converge to the ``usual'' optima. At the same time, we argue that nonlocal learning is possible in situations where standard calculus fails. As a stylized numerical example of this, we consider the problem of non-differentiable parameter estimation on a non-smooth translation manifold and show that our \emph{nonlocal} gradient descent recovers the unknown translation parameter from a non-differentiable objective function.

\end{abstract}

\section{Introduction}\label{sec:intro}
Nonlocality is an important paradigm in understanding the behavior of a system $X$ undergoing spatio-temporal or structural changes, particularly when such changes are abrupt or singular. A system exhibits \emph{nonlocality} if the changes occurring at some point $x\in X$ are influenced by the interactions between the $x$ and its neighbors. In contrast, a \emph{local} phenomenon is one where changes at $x$ depend solely on $x$. In order to perform predictive analysis on nonlocal systems, an important step is to develop nonlocal analytic/computational tools i.e., \emph{nonlocal calculus} and corresponding nonlocal descriptive models.

A generic description of a \emph{nonlocal operator} is a mapping $F$ from some topological space $X$ into a set $Y$ such that the value of $F$ at a point $x \in X$ depends on not only $x$ but also points in some neighborhood $U(x)$. In contrast, for a \emph{local operator} $L$, its value at $x$ depends solely on $x\in X$. Some of these nonlocal analyses lend themselves to a multiscale viewpoint, where instead of a single nonlocal operator $F:X \rightarrow Y$, we have instead a family of such operators $F_n: X\rightarrow Y$ indexed by a scale parameter $n\in \mathbb{R}$. A common theme is the convergence of the  nonlocal operators $F_n$ to a local counterpart, under suitable assumptions. In other words, the scale parameter $n$ controls the nonlocality in such a way that, asymptotically, as $n \rightarrow\infty$, the nonlocal operators $F_n\rightarrow F^*$, where $F^*:X\rightarrow Y$ is a \emph{local} operator. The exact nature of convergence, i.e., the topology of the function spaces in question, naturally has an important role to play in such a situation.

The prototypical class of nonlocal operators is the class of integral operators. If $\Omega$ is an open subset of a Euclidean space, and $f:\Omega \rightarrow \mathbb{R}$ a suitable function, then we can consider the operator
\begin{equation}(Tf)(x) = \int_\Omega f(y)K(x,y)dy \end{equation}
for an appropriate kernel $K(\cdot,\cdot):\Omega \times \Omega \rightarrow \mathbb{R}$. A generic enough $K$ can capture the nonlocal interaction between points of $\Omega$, or said differently, the modeler has the flexibility of choosing the kernel $K$ generic enough to model the phenomena under investigation. We also can consider a sequence of such kernels $K_n$ which give rise to the integral operators \begin{equation}(T_nf)(x) = \int_\Omega f(y)K_n(x,y)dy .\end{equation} Assuming that the $K_n$ converge, in some sense, to a kernel $K$, we can conclude, with reasonable assumptions, the convergence of the corresponding operators $T_n$ to $T$ (see \cite{atkinson}). In case $K(x,y) = \delta(x-y)$, where $\delta(\cdot)$ is the Dirac mass at the origin, we have, at least formally, that the $T_n$ converge to the identity, i.e.,  $Tf(x) = f(x)$. This observation underpins most of our further discussions.

While classical, differential operator (gradient, Hessian etc.) based methods work well under benign conditions, they cannot be used with singular domains and/or functions, since partial derivatives with the required regularity do not exist. The nonlocal calculus we study here is based on integral operators. The integral operators of nonlocal models come equipped with so-called ``interaction kernels" and can be defined for irregular domains or non-differentiable functions. As such, they do not suffer from singularity issues to the extent of their differential counterparts (\cite{peri4,peri5,gunz2}) and have been applied in practical problems with great success. In addition to being applicable to singular problems, nonlocal operators can be shown to converge to their local counterparts under suitably natural conditions: nonlocal gradients (resp. Hessians) converge to the usual gradient (resp. Hessian) when the latter exists (see \cite{brezis,menspec,spechess}).

\subsection{Contributions}
Our primary contributions are the theory and algorithmic applications of nonlocal differential operators, mainly nonlocal first and second order (gradient and Hessian) operators to perform optimization tasks. Specifically, we develop nonlocal analogs of gradient (including stochastic) descent, as well as Newton's method. We establish convergence results for these methods that show how the optima obtained from nonlocal optimization converge to the standard gradient/Hessian based optima, as the nonlocal interaction effects become less pronounced. While our definition of nonlocal optimization methods is valid for irregular problems, our analysis will assume the required regularity of the objective functions in order to compare the local and nonlocal approaches. As a stylized numerical application of our nonlocal learning framework, we apply our nonlocal gradient descent procedure to the problem of parameter estimation on a non-differentiable ``translation" manifold as studied via multiscale smoothing in \cite{wakin1} and \cite{wakin2}. In our framework, this problem can be solved without resorting to a multiscale analysis of the manifold; instead, the nonlocal gradient operators we use can be viewed as providing the required smoothing. While we have emphasized the use of nonlocal operators to non-differentiable optimization as an alternative to traditional subgradient methods (\cite{shor}), the scope of the theory is much wider. Indeed, one may use nonlocal methods to model long range interactions (\cite{interaction}), phase transitions (\cite{phase}) etc.


\subsection{Prior Work}
Nonlocal models have, in recent years, been fruitfully deployed in a wide variety of application areas: in image processing (\cite{osher}) nonlocal operators are used for total variation (TV) inpainting and also image denoising (\cite{nlm1,nlm2,nlm3,nlm4}) with the nonlocal means framework. Computational/engineering mechanics has profited greatly in recent years with a plethora of nonlocal models. Plasticity theory (\cite{plastic}) and fracture mechanics (\cite{fracture,creep}) are two prominent examples. Indeed, the field of peridynamics (see \cite{peri1,peri2,peri3}) is a prime consumer of nonlocal analysis. Nonlocal approaches have also featured in crowd dynamics \cite{crowd}, and cell biology (\cite{cells}). A nonlocal theory of vector calculus has been developed along with several numerical implementations (see \cite{gunz1,gunz3,dubook}). The nonlocal approach has also found its way into deep learning (\cite{dnll}). On the theoretical side, nonlocal and fractional differential operators have been studied in the context of nonlocal diffusion equations (\cite{diff} and references therein) as well as purely in the context of nonlocal gradient/Hessian operators (\cite{menspec, spechess, brezis, grad1, grad2}). The use of fractional calculus via Caputo derivatives and Riemann-Liouville fractional integrals for gradient descent (GD) and backpropagation has been considered in \cite{wei,pu,chen,sheng,wang}. Our work uses instead the related notion of nonlocal calculus from \cite{menspec,spechess,brezis,dubook} as its basis for developing not only a nonlocal (stochastic) GD, but also a nonlocal Newton's method, and is backed by rigorous convergence theory. In addition, our nonlocal stochastic GD approach interprets, in a very natural way, the nonlocal interaction kernel as a conditional probability density. The multiple definitions of fractional and nonlocal calculi are related (see for e.g. \cite{karni1}).

\subsection{Notation and Review}\label{sec:intro1.1}
In all that follows, $\Omega \subset \mathbb{R}^D$ will denote a simply connected bounded domain (i.e., a bounded, connected, simply connected, open subset of $\mathbb{R}^D$)  with a smooth, compact boundary $\partial \Omega$ where $D$ is a fixed positive integer. $\Omega$ will be endowed with the subspace topology of $\mathbb{R}^D$. We will consistently denote by $x,y,z$ points in $\Omega$. Subscripts of vectorial (resp. tensorial) quantities will denote the corresponding components. For e.g., $x_i$ or $A_{i,j}$ will denote the $i^\text{th}$ component of the vector $x$ or $(i,j)^{\text{th}}$ entry of matrix $A$ respectively for $i,j=1,\ldots,D$. For vectors $x,y\in \mathbb{R}^D$ the notation $x\otimes y \in \mathbb{R}^{D\times D}$ denotes the outer product of $x$ and $y$. The letters $i,j,k,n,m,N$ will denote positive integers while $t,s,p,\epsilon,\delta, M, K$ denote arbitrary real quantities. By $B_{\epsilon}(x)$, we mean an open ball centered at $x$ with radius $\epsilon > 0$. Furthermore, $u,v,f,g,\rho: \Omega \rightarrow \mathbb{R}$ will in general denote real-valued functions with domain $\Omega$. All integrals will be in the sense of Lebesgue, and a.e. means almost every or almost everywhere with respect to Lebesgue measure. Given an arbitrary collection of subsets $\mathcal{S}$ of $\mathbb{R}^D$, the $\sigma$-algebra generated by $\mathcal{S}$ will be denoted by $\sigma(\mathcal{S})$.

Given $u:\Omega \rightarrow \mathbb{R}$, we define the support $\text{supp}(u) \subset \Omega$ as \[
\text{supp}(u) = \overline{\{x \in \Omega: |u(x)| > 0 \text{ a.e.}\}},
\]
where the overbar denotes closure in $\mathbb{R}^D$. A \emph{multi-index} $\bar{n}=(n_1,\ldots,n_D)$ is an ordered $D$-tuple of non-negative integers $n_i, i=1,\ldots D$, and we let $|\bar{n}|:=\sum_{i=1}^{D}n_i$. For a suitably regular function $u$, we denote by \[\frac{\partial^{\bar{n}}u}{\partial x^{\bar{n}}}:=\frac{\partial^{n_1}}{\partial x^{n_1}}\frac{\partial^{n_2}}{\partial x^{n_2}}\ldots \frac{\partial^{n_D}u}{\partial x^{n_D}}.\] If we consider all multi-indices $\bar{n}$ with $|\bar{n}|=1$, we can identify $\frac{\partial^{\bar{n}}u}{\partial x^{\bar{n}}}$ with the usual gradient $\nabla u$ of $u$. In this case, we denote the $i^{\text{th}}$ component $\frac{\partial u}{\partial x_i}$ of $\nabla u$ as $u_{x_{i}}$ for $i=1,\ldots D$. Likewise, the Hessian of $u:=Hu$ at a point $x\in\Omega$ is the $D\times D$ matrix $(Hu(x))_{i,j} = \frac{\partial^2u}{\partial x_i \partial x_j}(x)$.
We shall use the following function spaces:
\begin{equation}
\begin{split}
L^p(\Omega,\mathbb{R}^{k})&=\{u:\Omega \rightarrow \mathbb{R}^k: \int_{\Omega}|u|^p < \infty \}, 1\le p < \infty\}, \\
L^{\infty}(\Omega)&=\{u:\Omega \rightarrow \mathbb{R}: \text{ess sup}(u) < \infty \} \text{ where ess sup}(u):=\text{ essential supremum of } u, \\
W^{m,p}(\Omega)&=\{u \in L^p(\Omega): \frac{\partial^{\bar{n}}u}{\partial x^{\bar{n}}} \in L^p(\Omega), 0\le|\bar{n}|\le m, \, 1\le p \le \infty\},\\
C^{m}(\Omega)&:= \{u:\Omega \rightarrow \mathbb{R}: u \text{ is $m$-times continuously differentiable}\},\\
C^{m}_{c}(\Omega)&:= \{u:\Omega \rightarrow \mathbb{R}: u \in C^{m}(\Omega) \text{ with compact support}\},\\
Lip(\Omega,M)&:= \{u:\Omega \rightarrow \mathbb{R}: u \text{ is Lipschitz continuous with constant $M$}.\}
\end{split}
\label{eq:A_fnspace}
\end{equation}
We let $L^p(\Omega,\mathbb{R}^{1}):=L^p(\Omega)$. Note that each of the above spaces $\mathcal{X}$ is a normed linear space. In particular $H^1(\Omega) := W^{1,2}(\Omega)$ is a Hilbert space with the inner product $(u,v)_{H^1(\Omega)}=(u,v)_{L^2(\Omega)}+(\nabla u,\nabla v)_{L^2(\Omega)}$. The subspace $H_{0}^{1}(\Omega)$ of $H^{1}(\Omega)$ consists of those $u\in H^{1}(\Omega)$ that vanish (i.e., have \emph{zero trace}) on the boundary $\partial \Omega$. Note that the space $C_{0}^{1}(\Omega)$ consisting of continuously differentiable functions vanishing on $\partial \Omega$ is dense in $H^{1}_{0}(\Omega)$ (see \cite{adams}). Next, $|\cdot|$ denotes the absolute value of real quantities, $\|\cdot\|$ will denote the Euclidean norm of $\mathbb{R}^D$ and $\|\cdot\|_{\mathcal{X}}$ will denote the norm in any of the function spaces $\mathcal{X}$ above.

Following \cite{menspec,brezis}, we will fix a sequence of radial functions $\rho_{n}:\mathbb{R}^D \rightarrow \mathbb{R}, n=1,\ldots $ which satisfy the following properties:
\begin{equation}
\begin{cases}
    \rho_n \geq 0,\\
     \int_{\mathbb{R}^D}\rho_n(x)dx =1, \\
      \lim_{n\rightarrow \infty}\int_{\|x\|>\delta}\rho_n(x)dx =0, \, \forall \delta > 0.\\
\end{cases}\label{eq:rho}
\end{equation}
Note that by \emph{radial} we mean that $\rho_{n}(x) = \rho_{n}(Ox)$ for any orthogonal matrix $O:\mathbb{R}^D\rightarrow \mathbb{R}^D$ with unit determinant. Alternatively, this means that $\rho_{n}(x) = \hat{\rho}_{n}(\|x\|)$ for some function $\hat{\rho}_{n}: \mathbb{R} \rightarrow \mathbb{R}$. We can specify (but will not insist) that the $\rho_{n}$ have compact support or even be of class $C^{\infty}_{c}(\Omega)$. Note that for our analysis of the nonlocal stochastic gradient descent method, we shall interpret the $\rho_{n}$ as a sequence of probability densities on $\mathbb{R}^D$ that approach the Dirac density $\delta_0$ centered at the origin. Prototypical examples of $\rho_{n}$ are given by sequences of Gaussian distributions with increasing variance (the non-compact support case) and the so-called bump functions with increasing height (the compact support case). In both these cases, we can show that the sequence of $\rho_n$ converge as Schwartz distributions to the Dirac density $\delta_0$. Details can be found in \cite{evans}. Finally, the set of all radial probability densities on $\mathbb{R}^D$, i.e., those $\rho:\mathbb{R}^D \rightarrow \mathbb{R}$ that satisfy the first two conditions of equation \ref{eq:rho} will be denoted by $\mathcal{P}$.
\subsection{Nonlocal Operators}
Let us now focus on the main objects of study in this paper, \emph{nonlocal} differential operators. We note that there is no unique notion of ``the" nonlocal gradient (resp. Hessian). Indeed, differing choices of kernels may give rise to differing notions of gradients (resp. Hessians).
\subsubsection{First Order Operators}
As done in \cite{menspec}, for any radial density $\rho\in \mathcal{P}$, we can consider the linear operators $\frac{\partial_{\rho}}{\partial_{\rho}x_i}$ for $i=1,\ldots, D$ defined as:\begin{equation}\frac{\partial_{\rho}u}{\partial_{\rho}x_i}(x):= D\lim_{\epsilon \rightarrow 0} \int_{\Omega\cap B_{\epsilon}^{c}(x)}\frac{u(x)-u(y)}{\|x-y\|}\frac{x_i-y_i}{\|x-y\|}\rho(x-y)dy,\label{eq:nlder1}\end{equation}
where the domain of $\frac{\partial_{\rho}}{\partial_{\rho}x_i}$ is the set of all functions for which the principle value integral on the right hand side of equation \ref{eq:nlder1} exists (see \cite{menspec}). 

Stacking the $\frac{\partial_{\rho}u}{\partial_{\rho}x_i}$ as a vector yields a \emph{nonlocal} gradient: \begin{equation}\nabla_{\rho}u(x):= D\lim_{\epsilon \rightarrow 0} \int_{\Omega\cap B_{\epsilon}^{c}(x)}\frac{u(x)-u(y)}{\|x-y\|}\frac{x-y}{\|x-y\|}\rho(x-y)dy,\label{eq:nlder2}\end{equation}
in other words $\frac{\partial_{\rho}u}{\partial_{\rho}x_i}(x) = (\nabla_{\rho}u(x))_i$ for $i=1,\ldots,D$. Purely formally, if we allow for $\rho = \delta_0$, the Dirac delta distribution centered at $0$, then the ``nonlocal" gradient defined above coincides (in the sense of distributions) to the usual ``local" gradient. However, we would then be operating in the highly irregular realm of distributions, far removed from the space of functions that are more in line with the applications in mind. Nevertheless, much of the analysis in this paper revolves considering the case of a ``well-behaved" $\rho$ approximating $\delta_0$ so that $\nabla_{\rho}u$ approximates the true gradient $\nabla u$.

If $u\in C^{1}(\bar{\Omega})$, or more generally, as \cite{menspec} note (see also \cite{brezis}), if $u\in W^{1,p}(\Omega)$, the map \[y \mapsto \frac{|u(y)-u(x)|}{\|x-y\|}\rho(x-y)\in L^1(\Omega)\] for a.e. $y\in \Omega$.
In these cases, \emph{the principle value integral in equation \ref{eq:nlder2} exists and coincides with the Lebesgue integral} of $\frac{u(x)-u(y)}{\|x-y\|}\frac{x-y}{\|x-y\|}\rho(x-y)$ over all of $\Omega$. Thus, if $u\in C^{1}(\bar{\Omega})$ (or $u\in W^{1,p}(\Omega)$), then we may drop the limit in equation \ref{eq:nlder2} and write:
\begin{equation}\nabla_{\rho}u(x):= D\int_{\Omega}\frac{u(x)-u(y)}{\|x-y\|}\frac{x-y}{\|x-y\|}\rho(x-y)dy.\end{equation}
At times, we may wish to evaluate the integral on a subset $\Omega^{*}\subset \Omega$. In this case, we write $\nabla_{\rho}^{\Omega^{*}}u(x)$ to denote the corresponding nonlocal gradient:
\begin{equation}\nabla_{\rho}^{\Omega^{*}}u(x):= D \int_{\Omega^{*}}\frac{u(x)-u(y)}{\|x-y\|}\frac{x-y}{\|x-y\|}\rho(x-y)dy.\label{eq:A_anlder3}\end{equation}
If $\Omega^{*}=\Omega$, we simply denote this as $\nabla_{\rho}u(x)$.

The nonlocal gradient $\nabla_\rho u$ defined via $\rho$ above exists for $u\in W^{1,p}$, which consists of functions that are only \emph{weakly} differentiable and for which the standard (strong) gradient $\nabla u$ does \emph{not} exist.

Specializing to the fixed sequence $\rho_n$ defined in equation \ref{eq:rho}, we write $\frac{\partial_{n}u}{\partial_{n}x_i}(x) :=\frac{\partial_{\rho_n}u}{\partial_{\rho_n}x_i}(x)$ or, more succinctly, as $u^{n}_{x_i}$, and similarly $\nabla_n u(x) :=\nabla_{\rho_n}u(x)$. We refer to the subscript $n$ as the ``scale of nonlocality". It is easy to verify that unlike the usual local gradient $\nabla(\cdot)$, the nonlocal $\nabla_n(\cdot)$ does not satisfy the product rule. The mode and topology of convergence of the sequence $\nabla_n u$ to $\nabla u$ as $n\rightarrow \infty$ has been investigated by many authors (see \cite{osher,gunz1,menspec})
\subsubsection{Second Order Operators}
Just as a first order nonlocal gradient operator has been defined earlier, we can define second order nonlocal differential operators as well. We consider \emph{nonlocal Hessians}. In this case, even for a fixed scale of nonlocality $n$ and associated density (kernel) $\rho_n(\cdot)$, one has several choices for defining a nonlocal Hessian. We list below four possibilities, and specify the $(i,j)^{\text{th}}$ entry of nonlocal Hessian matrices for $i,j=1,\ldots,D$. For the time being, we assume enough regularity for $u$ in order for the Hessians to make sense. We will specialize to $u\in C^{2}_{c}(\Omega)$ and $u\in W^{2,p}(\Omega)$ in section \ref{sec:A_ahessian}. Recall the notation $u^{n}_{x_i}(x) = \frac{\partial_{n}u}{\partial_{n}x_i}(x)$. 
\begin{equation}
\begin{cases}
    (H_{n,m}^{1})_{i,j}(x) &:=D \int_{\Omega}\frac{u^{n}_{x_i}(x)-u^{n}_{x_i}(y)}{\|x-y\|}\frac{x_j-y_j}{\|x-y\|}\rho_m(x-y)dy := \frac{\partial_m u^{n}_{x_i}}{\partial_m x_j}(x) := (u_{x_i}^{n})_{x_j}^{m},\\
(H_{n}^{2})_{i,j}(x) &:=D \int_{\Omega}\frac{u_{x_i}(x)-u_{x_i}(y)}{\|x-y\|}\frac{x_j-y_j}{\|x-y\|}\rho_n(x-y)dy := \frac{\partial_n u_{x_i}}{\partial_n x_j}(x) := (u_{x_i})_{x_j}^{n},\\
(H_{n}^{3})_{i,j}(x) &:=\frac{\partial u^{n}_{x_i}}{\partial x_j}(x):=(u_{x_i}^{n})_{x_j},\\
(H_{n}^{4})_{i,j}(x) &:=\frac{D(D+1)}{2}\int_{\mathbb{R}^D}\frac{u(x+h)-2u(x)+u(x-h))}{\|h\|^{2}}\frac{(h\otimes h-\frac{\|h\|^2}{D+2} I_D)_{i,j}}{\|h\|^2}\rho_{n}(h)dh.\\
\end{cases}\label{eq:A_a4hessians}
\end{equation}
Of the four Hessians above, $H_{n}^{4}$ is explicitly defined and analyzed in great detailed by \cite{spechess}. It is interesting to note that $H_{n}^{4}$ is a non-local analog of the usual second-order central difference approximation to the second derivative (\cite{fd}).
\subsection{Paper Organization}After this introductory section, we consider $1^{\text{st}}$ order nonlocal optimization in section \ref{sec:1storder}. section \ref{sec:numapp} deals with a stylized numerical example of the $1^{\text{st}}$ order theory. In section \ref{sec:A_ahessian}, we develop the second order nonlocal optimization theory and the nonlocal analog of Newton's method. We conclude in section \ref{sec:conc}. The proofs of all theoretical results are provided in the supplementary appendix section \ref{sec:A_aintro}.
\section{$1^{\text{st}}$ Order Nonlocal Theory and Methods}\label{sec:1storder}
We now analyze the $1^{\text{st}}$ order nonlocal optimization theory and associated numerical methods. Recall that $\nabla_n u$ exists for $u \in W^{1,p}$, even if the classical gradient $\nabla u$ is not defined and we also know from \cite{menspec} that if, in addition, $u\in C^1({\bar{\Omega}})$ (so that the usual gradient $\nabla u$ also exists), then we have that $\nabla_n u \rightarrow \nabla u$ locally uniformly. Thus, if $x^* \in \Omega$ is a local minimizer of $u$, then $\nabla u(x^*) = 0$ and we have that $\nabla_n u(x^*) \rightarrow \nabla u(x^*) = 0 $, from which we can immediately conclude the following.
\begin{proposition}Let $x^{*} \in \Omega$ be a local minimum of $u\in C^{1}(\bar{\Omega})$. Then given an $\epsilon>0$, there is an $N>0$ such that $\|\nabla_n u (x^*)\|<\epsilon$ for all $n>N$.
\label{thm:1ordersmall}
\end{proposition}
While we shall not be using the following result directly in the development of nonlocal optimization methods, it is nonetheless an important analog of the local case. $\nabla_{n}^{\Omega^{*}}u$ refers to the restriction of $\nabla_n u$ to a measurable subset $\Omega^{*} \subset \Omega$ (see supplementary appendix for more details).
\begin{theorem}\label{thm:order1}Let $x^{*} \in \Omega$ be an a.e. local minimum of $u\in W^{1,p}(\Omega)$. For each $n\in \mathbb{N}$ there exists a measurable subset $\Omega^{*}=\Omega^{*}(n)\subset \Omega$ such that $\nabla_{n}^{\Omega^{*}}u(x^*)=0$.
\end{theorem}

\subsection{Nonlocal Taylor's Theorem}
Consider first the standard affine approximant \[A(x_0,x):=u(x_0)+(x-x_0)^T\nabla u(x_0)\] to $u(x)\in C^1(\Omega)$ at $x_0$ with the corresponding remainder term $r(x_0,x)=u(x)-A(x_0,x).$ The nonlocal analog of the affine approximant is \[A_n(x_0,x):=u(x_0)+(x-x_0)^T\nabla_n u(x_0)\] and the corresponding nonlocal remainder term is $r_n(x_0,x)=u(x)-A_n(x_0,x).$
Standard calculus dictates that the remainder term $r(x_0,x)=o(\|x-x_0\|)$, and by uniqueness of the derivative, we can conclude that the nonlocal remainder $r_n(x_0,x)$ will in general not be $o(\|x-x_0\|)$. Indeed, the standard (local) affine approximant is asymptotically the unique best fit linear polynomial and hence, the nonlocal affine approximant must necessarily be sub-optimal. It is therefore important to understand how $r_n(x_0,x)$ relates to $r(x_0,x)$ as $n\rightarrow \infty$, and this is the objective of the next result.
\begin{proposition}\begin{itemize}
\item Let $u\in C_{c}^1(\Omega)$. Then $A_n(x_0,x)\rightarrow A(x_0,x)$, and, $r_n\rightarrow r$ uniformly as $n\rightarrow \infty$.
\item Let $u\in H_{0}^1(\Omega)$. Then $\|r_n(\cdot,\cdot)- r(\cdot,\cdot)\|_{L^2(\Omega\times \Omega)}\rightarrow 0$ as $n\rightarrow \infty$.
\end{itemize}
\label{thm:taylor1}
\end{proposition}
\subsection{Nonlocal Gradient Descent}
We define the nonlocal gradient descent as follows. We first fix the scale of nonlocality $n$, and initialize $x^{0,n} \in \Omega$. We then perform the update:
\begin{equation}
   x^{k+1,n}=x^{k,n}-\alpha_{n}^{k}\nabla_n u(x^{k,n}).
\end{equation}
Although one can easily conceive of an adaptive procedure with $n$ varying with step $k$, we emphasize that for our analysis, the scale $n$ of nonlocality does not change between the steps of the descent. Henceforth, we denote by $\nabla u^k:=\nabla u(x^k)$ and $\nabla_n u^{k,n}:=\nabla_n u(x^{k,n})$. The step size $\alpha_{n}^{k}$ and initial point $x^{0,n}\in \Omega$ are chosen small enough to ensure that the iterates $x^{k,n}$ all remain in $\Omega$. We emphasize that nonlocal gradient descent is applicable so long as $\nabla_n u$ exists, in particular, even if $\nabla u$ does not exist.

\subsubsection{Suboptimal Stepsize}
Consider first the sub-optimal step size situation, i.e., we use the same step size for both the local and nonlocal descent methods. This stepsize does not have to be the one obtained by linesearch. We will later consider the more involved situation of a \emph{nonlocal stepsize}, or a stepsize that changes with $n$, the scale of nonlocality. We thus now consider the situation with $\alpha_{n}^{k}=\alpha^{k}$ and compare the two descent methods:
\begin{equation}
\begin{cases}
   x^{k+1}&=x^{k}-\alpha^{k}\nabla u^{k},\\
   x^{k+1,n}&=x^{k,n}-\alpha^{k}\nabla_n u^{k,n},\\
\end{cases}
\end{equation}
with the same initial condition $x^0=x^{0,n}$ for all $n$. Note that if $u\in Lip(\Omega,M)$, then we can conclude that the sequence of nonlocal iterates is bounded if the step sizes $\alpha^k$ are bounded, and decreasing fast enough:
\begin{proposition}If the step sizes $\alpha^k$ are chosen such that the sum $\sum_{m=1}^{K}\alpha^{m}<1$ for any $K>0$, then we can conclude that the sequence $\{x^{k,n}\}$ is bounded.
\label{thm:bdd}\end{proposition}
In particular, proposition \ref{thm:bdd} and the Heine-Borel Theorem imply that we are guaranteed (at least) a convergent subsequence of the sequence $\{x^{k,n}\}$. For the remainder of this subsection, we assume for purposes of our comparative analysis of the local and nonlocal gradient descent methods, that $u\in C^2(\Omega) \cap Lip(\Omega,M)$. Thus, the sequence of nonlocal iterates $\{x^{k,n}\}$ in this case are bounded, and hence, there is a compact $V\subset \Omega$ such that $x^{k,n} \in V$ for all $k$. We are interested in studying the quantity $\|x^{k}-x^{k,n}\|$ as $n$ increases. We have the following result.
\begin{theorem}
Let $u\in C^2(\Omega) \cap Lip(\Omega,M)$ and let $\epsilon > 0 $ be given. Assume that nonlocal gradient descent is initialized with $x^{0}$ for all $n$, and with the same sequence of step sizes $\alpha^k$ for all $n$ satisfying $\sum_{m=1}^{K}\alpha^{m}<1$ for any $K>0$. Then, for each $k$, there is an $N>0$ such that for all $n>N$, we have $\|x^{m}-x^{m,n}\| < \epsilon\,$ for $m=1,\ldots,k+1$.
\label{thm:suboptimalgd}\end{theorem}
Theorem \ref{thm:suboptimalgd} implies that, if the objective function admits a gradient, then for $n$ large enough, the iterates of nonlocal gradient descent get arbitrarily close to the iterates of the standard gradient descent. The important point to note is that while nonlocal gradient descent may be applicable in situations where the classical gradient is unavailable, we would not have a convergence analysis comparing the nonlocal gradient descent with the local analog. However, even in such cases, as we noted earlier, proposition \ref{thm:bdd} guarantees that we have a convergent subsequence of iterates of the nonlocal gradient descent method.

\subsubsection{Optimal Stepsize}
We now consider the case of optimal stepsize obtained by linesearch. We fix a maximal stepsize range of $I=[0,A]$ for some $A>0$. As with the suboptimal stepsize case, we assume that $u\in C^2(\Omega) \cap Lip(\Omega,M)$ so that the sequence of nonlocal iterates $x^{k,n}$ are bounded. Fix a compact $V\subset \Omega$ such that $x^{k,n} \in V$ for all $k$. In the classical gradient descent approach, exact linesearch for the optimal stepsize at iteration $k$ involves finding $\alpha^{k}=\text{argmin}_{\alpha\in I}\,u(x^k-\alpha\nabla u^k)$, i.e., we perform:
\begin{equation*}
   x^{k+1}=x^{k}-\alpha^{k}\nabla u^{k},
\end{equation*}
where 
\begin{equation*}
   \alpha^{k}=\text{argmin}_{\alpha\in I}\,u(x^k-\alpha\nabla u^k).
\end{equation*}
We adapt it now to the nonlocal case. The steps of nonlocal gradient descent with exact linesearch for the optimal stepsize takes the form:
\begin{equation}
   x^{k+1,n}=x^{k,n}-\alpha^{k,n}\nabla_n u^{k,n},
\end{equation}
where 
\begin{equation}
   \alpha^{k,n}=\text{argmin}_{\alpha\in I}\,u(x^{k,n}-\alpha\nabla_n u^{k,n}),
\end{equation}
and again we assume that $x^0=x^{0,n}$ for all $n$. Note that $A$ must be defined in such a way that $x^{k,n}\in \Omega$ for all $k$. Also, let us assume for simplicity that $\alpha^{k,n}$ is unique.
\begin{theorem}Let $u\in C^2(\Omega) \cap Lip(\Omega,M)$. Assume that the optimal stepsize version of nonlocal gradient descent is initialized with $x^{0}$ for all $n$. Then, for each $k$, there exists a subsequence $\{\alpha^{k,n_{m_1}(k)}\}$ of $\{\alpha^{k,n}\}_{n=1}^{\infty}$ such that $\alpha^{k,n_{m_1}(k)} \rightarrow \alpha^{k}$ as $n_{m_1}(k)\rightarrow \infty$ and a subsequence $\{x^{k,n_{m_2}(k)}\}$ of $\{x^{k,n}\}_{n=1}^{\infty}$ such that $x^{k,n_{m_2}(k)} \rightarrow x^{k}$ as $n_{m_2}(k)\rightarrow \infty$.\label{thm:optimalgd}\end{theorem}

\subsection{Nonlocal Stochastic Gradient Descent}
In this section, we consider the nonlocal version of stochastic gradient descent (SGD). A generic descent algorithm with a proxy-gradient $g^k \in \mathbb{R}^D$:
\begin{equation}\label{eq:sgd2}
\begin{split}
x^{k+1}&=x^{k}-\alpha^k g^k.\\
\end{split}
\end{equation}
After $K$ iterations, we output the averaged vector $\bar{x}=\frac{1}{K}\sum_{k=1}^{K}x^{k}$. If the \emph{expected value} of $g^k$ equals the true gradient $\nabla u^k$, then, roughly, the difference between the \emph{expected value} of $u(\bar{x})$ and the true minimum $u(x^{*})$ approaches $0$ as the number of iterations approaches $\infty$. This assumes that we sample $g^k$ from a distribution such that $\mathbb{E}(g^k|x^{k}) = \nabla u(x^k)$, or, more generally, $\mathbb{E}(g^k|x^{k}) \in \partial u(x^k)$, where $\partial u(x^k)$ is the subdifferential set of $u$ at $x^k$ (see \cite{shai}).

As motivation for our \emph{nonlocal} stochastic gradient descent analysis, let $X,Y$ be $\Omega-$valued random variables, with a conditional density $p_{Y|X}(y|x)=\rho(x-y).$ Here, $\rho(\cdot)$ is a radial density in the class $\mathcal{P}$ defined in section \ref{sec:intro1.1}. Note that since we interpret $\rho(x-y)$ to be $p_{Y|X}(y|x)$, we can \emph{choose} a density $p_{X}(x)$ and arrive at a joint distribution $p_{X,Y}(x,y)=\rho(x-y)p_{X}(x)$, and upon integration, arrive at the marginal distribution $p_{Y}(y)$ of $Y$. Let us denote by \begin{equation}k_u(x,y):=\frac{u(x)-u(y)}{\|x-y\|}\frac{x-y}{\|x-y\|}.\end{equation}If we let $\mathbb{E}_{\rho}[\cdot]$ denote the expectation with respect to the density $\rho$, then it is evident that \begin{equation}\nabla_{\rho} u(x) = \int k_u(x,y)\rho(x-y)dy :=  \mathbb{E}_{\rho}[k_{u}(X,Y)|X=x],\end{equation}or if we interpret $\nabla_{\rho}u(X)$ as a random variable, we can write $\nabla_{\rho} u(X) = \mathbb{E}_{\rho}[k_{u}(X,Y)|X].$  We will set $g^k=k_u(x^k,y)$, where $y$ is sampled from $p_{Y}(y)$ and perform the standard SGD with the knowledge that $\mathbb{E}_{\rho}[g^k|x^k] = \nabla_{\rho}u(x^k)$. We can therefore consider a sequence of nonlocal SGD with kernels $\rho_n(\cdot)$, and as $n\rightarrow \infty$, we prove that the nonlocal optima converge to the usual SGD optima. In the following, we assume $u$ is a convex, Lipschitz function with Lipschitz constant $M$.
\begin{lemma}\label{thm:sgd1}
Assume $u$ is convex. Given $\epsilon>0$, there is an $N>0$, such that for all $n>N$, we have \[u(y)- u(x) \ge (y-x)^T\nabla_n u(x) -\epsilon.\]
\end{lemma}
Lemma \ref{thm:sgd1} is a motivation of the following (well-known, see \cite{eps1,eps2,eps3,eps4}) generalization of the notion of the subderivative set of a function. 
\begin{definition}Let $\epsilon>0$ be given. A vector $z\in \mathbb{R}^D$ is an $\epsilon-$subgradient of $u$ at $x$ if $u(y)-u(x)\ge (y-x)^T z-\epsilon$ for all $y\in \Omega$. The collection of all $\epsilon-$subgradients of $u$ at $x$ is called the $\epsilon-$subdifferential set of $u$ at $x$ and is denoted by $\partial_{\epsilon}u(x).$
\end{definition}
It is obvious from the definition of $\partial_{\epsilon}u(x)$ and Lemma \ref{thm:sgd1} that given an $\epsilon>0$, there is an $N>0$ such that for all $n>N$, we have that $\nabla_n u(x) \in \partial_{\epsilon}u(x),$ and in particular, $\nabla u(x) \in \partial_{\epsilon}u(x)$ for any $\epsilon>0.$ If we choose to run SGD with an $\epsilon-$subgradient $g^k$ in place of a usual subgradient, we call this an \emph{$\epsilon$-stochastic subgradient descent ($\epsilon$-SGD)} algorithm.

\begin{theorem}Let $u$ be a convex, $M-$Lipschitz function and let $x^{*}\in \arg\min_{x:\|x\|\le B}u(x).$ Let $\epsilon>0$ be given. If we run an $\epsilon-$ subgradient descent algorithm on $u$ for $K$ steps with $\alpha=\sqrt{\frac{B^2}{M^2 K}}$ then the output vector $\bar{x}$ satisfies \[u(\bar{x}) -u(x^{*})\le \frac{BM}{\sqrt{K}} + \epsilon.\] Furthermore, for every $\hat{\epsilon}>\epsilon$, to achieve $u(\bar{x})-u(x^*)\le \hat{\epsilon},$ it suffices to run the SGD algorithm for a number of iterations that satisfies \[K\ge \frac{B^2M^2}{(\hat{\epsilon}-\epsilon)^2}.\]\label{thm:sgdsgd1}
\end{theorem}
We come now to the main result of this section. The result and proof are modeled after Theorem 14.8 of \cite{shai}.
\begin{algorithm}
 \caption{$\epsilon-$ Stochastic Subgradient Descent}
  \begin{algorithmic}[1]
   \Inputs{Scalar $\alpha, \epsilon>0$, integer $K>0$}
    \Initialize{\strut$x^{1} \gets 0$}
    \For{$k = 1, \ldots, K$}
      \State Choose $g^k$ at random from a distribution such that $\mathbb{E}[g^k|x^k]\in \partial_{\epsilon}u (x^k)$
      \State $x^{k+1} \gets x^{k}-\alpha g^k$
    \EndFor
    \State Output $\bar{x}=\frac{1}{K}\sum_{k=1}^{K}x^k$
  \end{algorithmic}
\end{algorithm}

\begin{theorem}\label{thm:sgdsgd2}
Let $B,M>0$. Let $u$ be a convex function and let $x^{*}\in \arg\min_{x:\|x\|\le B}u(x).$ Let $\epsilon>0$ be given. Assume that $\epsilon-SGD$ is run for $K$ iterations with $\alpha=\sqrt{\frac{B^2}{M^2K}}$. Assume also that for all $k$, we have $\|g^k\|\le M$ with probability $1$. Then, \[\mathbb{E}[u(\bar{x})]-u(x^{*})\le \frac{BM}{\sqrt{K}} + \epsilon.\] Thus, running $\epsilon-SGD$ for $K$ iterations with $K\ge \frac{B^2M^2}{(\hat{\epsilon}-\epsilon)^2}$ ensures that $\mathbb{E}[u(\bar{x})]-u(x^{*})\le \hat{\epsilon}.$
\end{theorem}
Note that Theorem \ref{thm:sgdsgd2} shows that the estimate of $\mathbb{E}[u(\bar{x})]-u(x^{*})$ is composed of two parts: $\frac{BM}{\sqrt{K}}$ corresponds to the usual SGD estimate as in \cite{shai} while the $\epsilon$ term corresponds to the effects of nonlocal interactions. We can now specialize the general $\epsilon-SGD$ theory above to the case of the nonlocal gradient operator. Recall our observation that given an $\epsilon>0$, we have $\nabla_n u(x) \in \partial_{\epsilon} u(x)$ for $n$ large enough. Thus, we can set $g^k = k_{u}(x^k,y)$ to be the $\epsilon-$ subgradients used in $\epsilon-SGD$ and choose $p_{Y}(y) = \rho_n(x-y)p_{X}(x)$ to be the distribution from which the samples $y$ are chosen. We then have a \emph{nonlocal $\epsilon-SGD$ procedure}.

\section{Stylized Application: Non-differentiable Parameter Estimation}\label{sec:numapp}
As a numerical example of the theory developed in section \ref{sec:1storder}, we consider the case of \emph{parameter estimation} of non-differentiable parametric mappings motivated by \cite{wakin1,wakin2} but restrict our attention to the case of a compact $1$ dimensional Euclidean parameter space $\Theta = [0,1]$. Let $\mathcal{M}$ be the $1$ dimensional manifold consisting of translations of a rectangular pulse as shown in figure \ref{fig:pulse}. $\mathcal{M}$ is therefore the $1$ dimensional manifold $\{f(x-\theta):\theta \in \Theta\}$ viewed as a submanifold of $L^2([0,1])$ and $f(x)$ is the template pulse in figure \ref{fig:pulse} (green dashed line with support $[0,0.125]$). Note that while $\Theta$  is a flat Euclidean space, its image $\mathcal{M} \subset L^2([0,1])$ is a topological manifold. It is known (see \cite{grimes,wakin1}) that if $\mathcal{M}$ is endowed with the $L^2([0,1])$ metric, then the map $\theta\mapsto f(x-\theta):=f_\theta$ is non-differentiable. Indeed \cite{grimes,wakin1} observe that\begin{equation}\frac{\|f_{\theta_{1}}-f_{\theta_{2}}\|_{L^2([0,1])}}{|\theta_1-\theta_2|}\propto C|\theta_1-\theta_2|^{-\frac{1}{2}}.\label{eqn:estimate1}\end{equation} Given a template point $g\in \mathcal{M}$ corresponding to an unknown $\theta^*$ the objective is to recover an estimate $\hat{\theta}$ of $\theta^*$ by minimizing the function $\mathcal{E}(\theta) = \|f_\theta-g\|_{L^2([0,1])}$. The approach taken in \cite{wakin1} is to ``smoothen" the manifold by convolution with a Gaussian kernel and applying the standard Newton's method to the smoothed manifold. We can directly address this problem without smoothing the manifold. Indeed, it is clear from the estimate \ref{eqn:estimate1} that $\frac{|\mathcal{E}(\theta_1)-\mathcal{E}(\theta_2)|}{|\theta_1-\theta_2|} \rho_n(\theta_1-\theta_2)$ is integrable as a function of $\theta_2$ for any $\theta_1\neq \theta_2$, and hence the nonlocal gradient of $\mathcal{E}(\theta)$ exists. We can therefore use our nonlocal gradient descent to this problem to arrive at an estimate $\hat{\theta}$ of $\theta^*$. Note that we can view the scale of nonlocality $n$ as playing the role of a scale parameter vis-a-vis the approach of \cite{wakin1}, however, our approach can be viewed as smoothing the tangent space, rather than the manifold itself. We apply our ``vanilla" nonlocal gradient descent with a learning rate $\alpha \in (0,1]$. Experiments were done on a 64-bit Linux laptop with $3.8$ GB RAM and 2.60GHz Intel Core i5 CPU. We consider the case of both bump function and Gaussian kernels that serve as $\rho_n$ in the nonlocal gradient definition, with $n=1,2,3$ (see figure \ref{fig:num}). The initial and ``ground truth" pulses corresponding to $\theta^0=0.1$ and $\theta^*=0.5$ are shown in figure \ref{fig:num}. Given the extreme nonconvexity of $\mathcal{E}(\theta)$, we adapted the nonlocal gradient descent presented earlier by decrementing the learning rate by half whenever the ratio of successive gradients between iterations exceeds $2.5$. The convergence for both bump function and Gaussian kernels is rapid as seen in figure \ref{fig:res}. The Gaussian kernel case gracefully converges: increasing $n$ results in more rapid convergence, as expected. However, the effect of numerical instability due to vanishing gradients is clear in the case of the bump function kernel: increasing $n$ results in nonmonotone convergence. Nevertheless, the efficacy of the nonlocal approach is apparent from this example, and extensions to higher dimensions are certainly possible, although the cost of dense quadrature-based integration like the one used here become prohibitive as the dimension of the parameter space increases.
\begin{figure*}[t!]
    \centering
\begin{subfigure}[t]{\textwidth}
        \centering
        \includegraphics[height=1.5in]{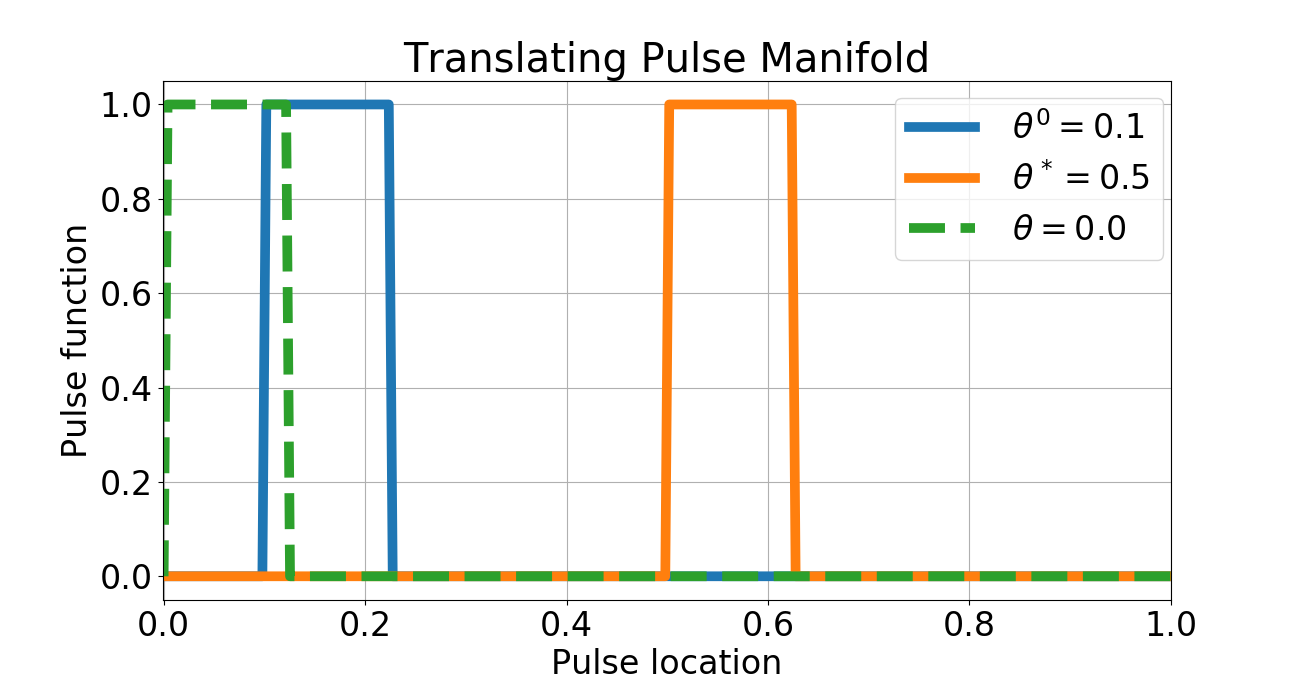}
        \caption{Points on $\mathcal{M}$}
    \end{subfigure}%
    ~ 
    \caption{}\label{fig:pulse}
\end{figure*}
\begin{figure*}[t!]
    \centering
    \begin{subfigure}[t]{0.5\textwidth}
        \centering
        \includegraphics[height=1.5in]{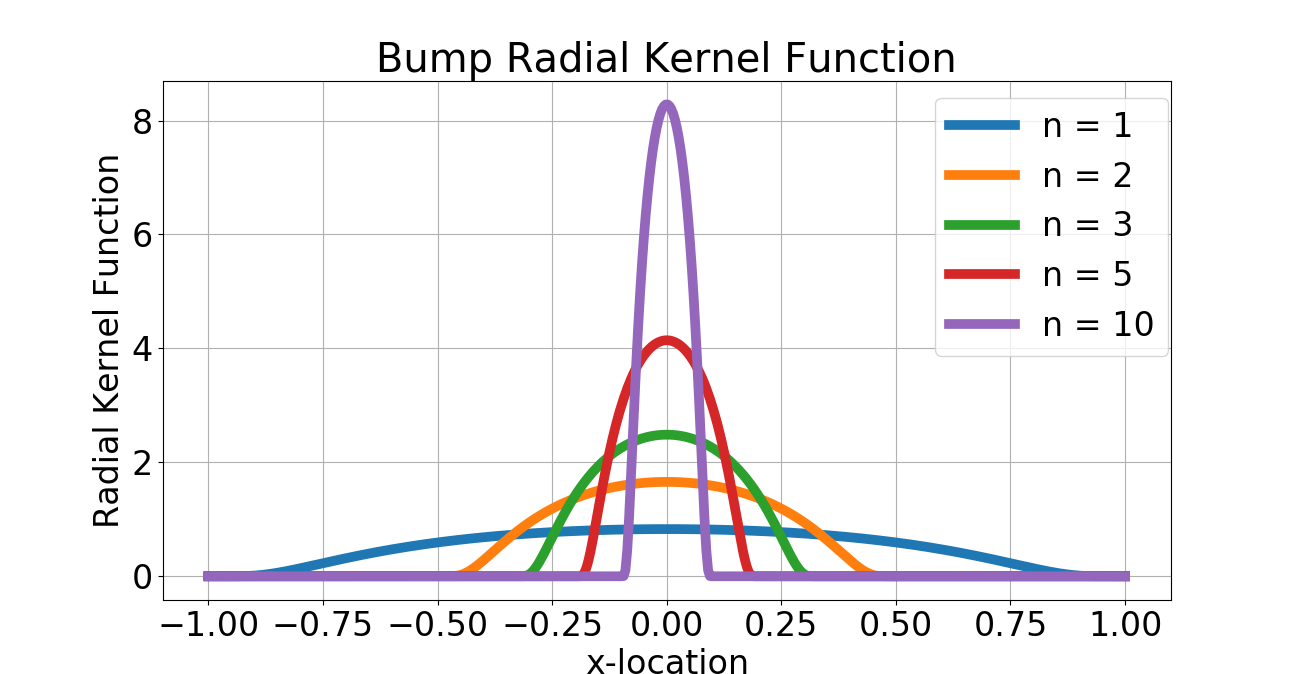}
        \caption{Bump function kernels}
    \end{subfigure}%
	~
    \begin{subfigure}[t]{0.5\textwidth}
        \centering
        \includegraphics[height=1.5in]{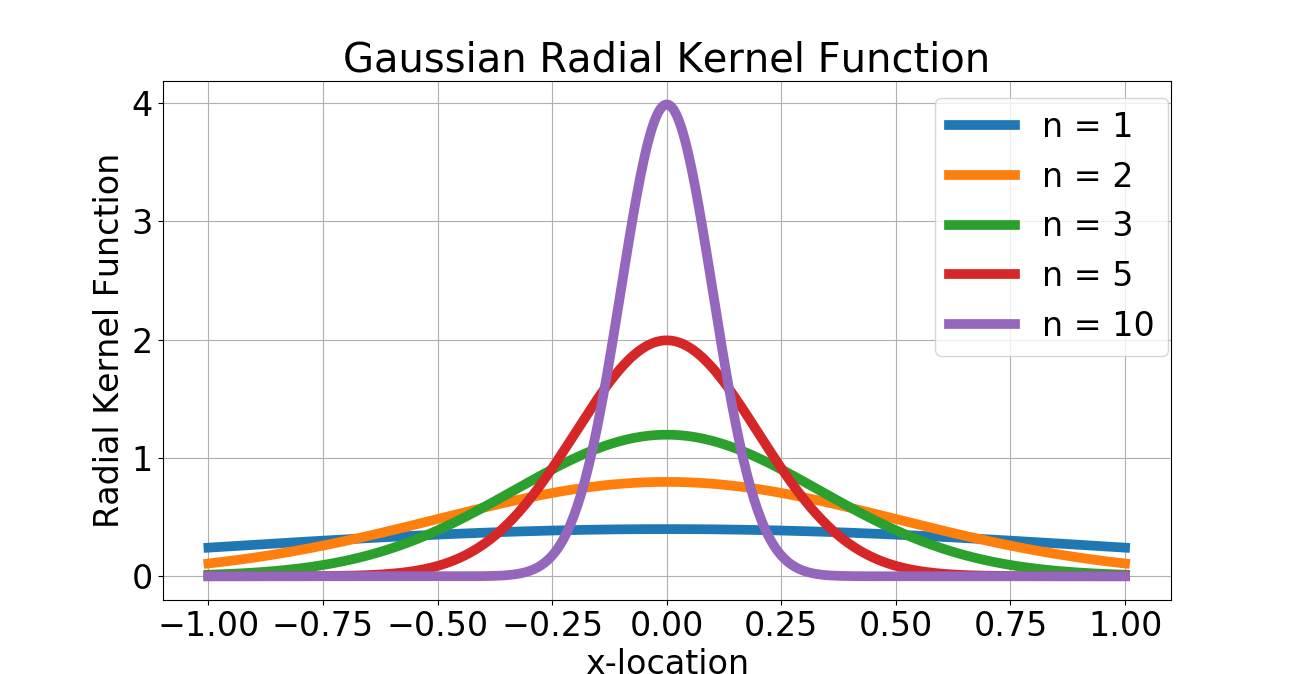}
        \caption{Gaussian kernels}
    \end{subfigure}    
    \caption{}\label{fig:num}
\end{figure*}
\begin{figure*}[t!]
    \centering
    \begin{subfigure}[t]{0.5\textwidth}
        \centering
        \includegraphics[height=1.55in]{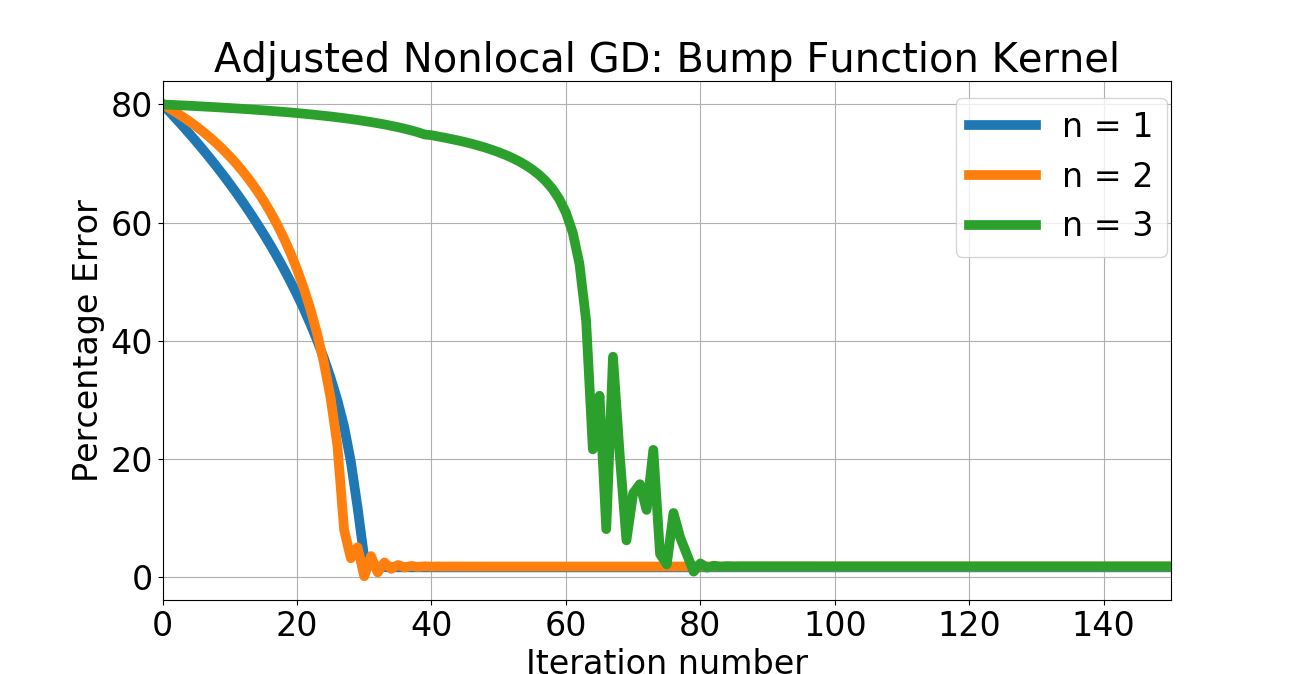}
        \caption{Convergence with bump function kernel}
    \end{subfigure}%
    ~ 
    \begin{subfigure}[t]{0.5\textwidth}
        \centering
        \includegraphics[height=1.55in]{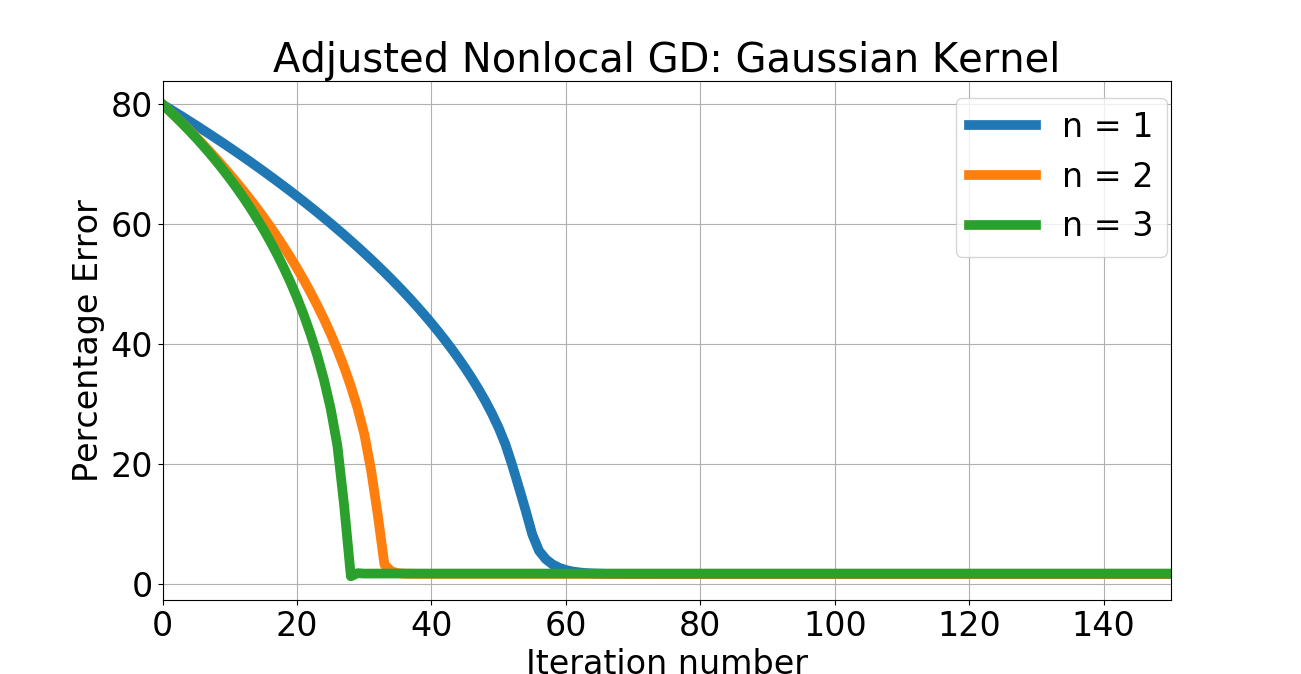}
        \caption{Convergence with Gaussian function kernel}
    \end{subfigure}
\caption{}\label{fig:res}
\end{figure*}
\section{$2^{\text{nd}}$ Order Nonlocal Theory and Methods}\label{sec:A_ahessian}
We now turn our attention to second order methods, in particular Newton's method. Recall from equation \ref{eq:A_a4hessians} that we have (at least) 4 distinct choices for defining a nonlocal analog of the Hessian of a function at a point. It is natural to consider the relationship between the four choices, and analyze the modes of convergence of each to the ``usual" local Hessian operator. The following results make this relationship clear.

\begin{theorem}
\begin{itemize}\item Let $u\in C^{2}(\bar{\Omega})$. Assume that there exists an $N>0$ such that for all $n>N$, $u_{x_i}^{n} \in C^{1}(\bar{\Omega})$. Then, For $n$ fixed and $n>N$, we have $(u_{x_i}^{n})_{x_j}^{m}\rightarrow (u_{x_i}^{n})_{x_j}$ locally uniformly as $m\rightarrow \infty,$ that is, $(H_{n,m}^{1})_{i,j} \xrightarrow{m\rightarrow\infty, \, n\,(\text{fixed}) >N }(H_{n}^{3})_{i,j}$ locally uniformly. If $u\in C^{2}_{c}(\Omega)$, the convergence is uniform.

\item Let $u\in C^{2}(\bar{\Omega})$. As $n\rightarrow \infty$, $ (u_{x_i})^{n}_{x_j} \rightarrow (u_{x_i})_{x_j} = \frac{\partial^2 u}{\partial x_i \partial x_j}$ locally uniformly, that is,  $(H_{n}^{2})_{i,j} \xrightarrow{n\rightarrow\infty}(H)_{i,j}$ locally uniformly. If $u\in C^{2}_{c}(\Omega)$, the convergence is uniform.\end{itemize}
\label{thm:hess1a}
\end{theorem}

We say a sequence $u_n \overset{H_{0}^{1}(\Omega)}\rightharpoonup u$ (read: $u_n$ converges weakly to $u$ in $H_{0}^{1}(\Omega)$) with $u_n, u \in L^2(\Omega)$ if $(u_n,v)_{L^2(\Omega)}\rightarrow (u,v)_{L^2(\Omega)}$ for all $v\in H^{1}_{0}(\Omega)$ as $n\rightarrow \infty$. We have the following theorem.
\begin{theorem}Let $u\in C^{2}(\Omega)$. Then: \begin{itemize} \item As $n\rightarrow \infty$, we have  $ (u_{x_i}^{n})_{x_j} \overset{H_{0}^{1}(\Omega)}\rightharpoonup (u_{x_i})_{x_j} = \frac{\partial^2 u}{\partial x_i \partial x_j}$, that is,  $ (H_{n}^{3})_{i,j} \overset{H_{0}^{1}(\Omega)}\rightharpoonup (H)_{i,j}$\item Combined with the hypothesis of theorem \ref{thm:hess1a}, we have:\[(H_{n,m}^{1})_{i,j} \xrightarrow{m\rightarrow\infty, \, n\,(\text{fixed}) >N }(H_{n}^{3})_{i,j} \overset{H_{0}^{1}(\Omega)}\rightharpoonup (H)_{i,j},\]\end{itemize}
with the convergence of $(H_{n,m}^{1})_{i,j}$ to $(H_{n}^{3})_{i,j}$ being locally uniform if $u\in C^2(\Omega)$ and uniform if $u\in C^2_{c}(\Omega)$.
\label{thm:hess2a}
\end{theorem}

The proofs of theorem \ref{thm:hess1a} and \ref{thm:hess2a} are found in section \ref{sec:A_aintro}. While we have shown weak convergence of $(H_{n,m}^{1})_{i,j} $ to $(H)_{i,j},$ as $m,n\rightarrow \infty$, we can show that under additional hypothesis on the radial functions $\rho_n$, we have an upper bound on how the sequence $(H^{1}_{n,m})_{i,j}$ deviates from strong (uniform) convergence.
\begin{theorem}Assume in addition to the existing hypotheses on $\rho_m$, that $\int_{\mathbb{R}^D}\frac{\rho_m(x)}{\|x\|}dx<M(m)$. Let $u\in C^2(\bar{\Omega})$. Then, given an $\epsilon >0$ there is an $N>0$ such that for all $n>N$, we have that \[\|(H_{n,m}^{1})_{i,j} - c_{j}^{m}(x)(H)_{i,j}\|_{L^{\infty}(\Omega)}\le 2D\epsilon \lim_{m\rightarrow \infty}M(m).\] Thus, if $M(m)$ is independent of $m$, we have local uniform convergence of $H_{n,m}^{1}$ to $H$. Otherwise, the $L^{\infty}$ error grows as $M(m)$.
\label{thm:hess3}
\end{theorem}
We note that theorem \ref{thm:hess3} provides a worst-case upper bound. Indeed, for the case of the Gaussian and ``bump" function densities we discussed earlier, a simple scaling argument shows that the corresponding constants $M(m)$ of theorem \ref{thm:hess3} grow with $m$, and therefore, the $L^{\infty}(\Omega)$ error is possibly unbounded. However, in practice, it seems plausible that there is a subclass of $C^{2}(\Omega)$ functions for which the $L^{\infty}(\Omega)$ error between $H_{n,m}^{1}$ and ``usual" local Hessian is bounded. We do not go further into the analysis of finding the optimal class of functions for which the error is bounded.
\subsection{Nonlocal Newton's Method}
We now consider a nonlocal version of Newton's method. Recall the standard Newton iterations to minimize $u(x)$ for $x\in \Omega$ with stepsize $\beta^k$:
\begin{equation}
   x^{k+1}=x^{k}-\beta^k (Hu^k)^{-1} \nabla u^k,
\end{equation}
where $Hu^k = Hu(x^k)$ is the usual Hessian of $u$ evaluated at $x^k$, and, as before, $\nabla u^k = \nabla u (x^k)$ is the gradient of $u$ evaluated at $x^k$. We shall consider the case of $\beta^k$ having been chosen by some line-search type method, sufficient to ensure a descent property at each step. The important point is that we consider the same $\beta^k$ for the nonlocal case:
\begin{equation}
   x^{k+1,n}=x^{k,n}-\beta^k (H_n u^{k,n})^{-1} \nabla_n u^{k,n}
\end{equation}
where $H_n$ is the nonlocal Hessian $H^{4}_{n}$ defined earlier, and $x^0=x^{0,n}$ for all $n$. As we did for gradient descent, in order to compare the nonlocal and local approaches, we will assume that the functions we wish to minimize are sufficiently regular. Again, we note that the nonlocal methods are still valid for possibly non-differentiable functions, so long as the nonlocal operators are well defined. We choose to do our analysis with $H^{4}_{n}$; while we can certainly consider the other possible nonlocal Hessians for numerical computations, we cannot guarantee theoretical convergence properties, due to only weak convergence of the other Hessians to the standard Hessian. Also, by \cite{spechess}, we can ensure uniform convergence of $H^{4}_{n}$ to $H$, though we tacitly expand the domain of $u$ from $\Omega$ to all of $\mathbb{R}^D$. Thus, if $u$ has compact support $K\subset \Omega$, we extend the domain of $u$ from $\Omega$ to all of $\mathbb{R}^D$ by zero while maintaining the regularity of $u$ in the extension to all of $\mathbb{R}^D$. We continue to denote by $u$ the extension of $u$ to all of $\mathbb{R}^D$.

Suppose $x^{*}\in \Omega$ is a minimizer of $u(x)$. We assume that $Hu(x)$ is continuous, i.e., $u(x)$ is at least $C^2(\mathbb{R}^D)$. Also, we assume $(Hu(x^{*}))^{-1}$ exists, and therefore, by \cite{zak,russell}, there is a neighborhood $U$ of $x^{*}$ such that $(Hu(x))^{-1}$ exists on $U$ and is continuous as $x^{*}$. We let $f(x)=(Hu(x))^{-1}\nabla u(x)$, on $U$ and assume enough regularity of $u$ to allow $f\in C^{1}_{c}(U)$. We also assume that $(Hu(x))^{-1}$ is uniformly continuous on $U$. We can now state the following result, whose proof can be found in section \ref{sec:A_aintro}.

\begin{theorem}\label{thm:newton1}
Assume that $u$ has compact support $K \subset \mathbb{R}^D$, that $(Hu(x))^{-1}$ is uniformly continuous on a neighborhood $U$ of $x^{*}$, and enough regularity of $u$ to allow $f\in C^{1}_{c}(U)$. Then, given $\epsilon>0$, there exists an $N>0$ so that for $n>N$, we have \[\|x^{k}-x^{k,n}\|<\epsilon.\]
\end{theorem}
We can now give a comparison between the usual Newton's method's quadratic convergence with the nonlocal counterpart.
\begin{theorem}\label{thm:newton2}
Assume the hypotheses of theorem \ref{thm:newton1}. Then, given $\epsilon>0$, there exists an $N>0$ so that for $n>N$, and $k$ large enough, we have \[\|x^{*}-x^{k,n}\|<M\|x^{*}-x^{k-1}\|^2+\epsilon,\]for some $M$ depending on the regularity of $u$.
\end{theorem}
\begin{proof}
We write $\|x^{*}-x^{k,n}\|\le\|x^{*}-x^{k}\|+\|x^{k}-x^{k,n}\|$. By the standard quadratic convergence properties of Newton's method (\cite{zak}), we have $\|x^{*}-x^{k}\|\le M\|x^{*}-x^{k-1}\|^2$ for suitable $M$ depending on the regularity of $u$. Next, by \ref{thm:newton1}, there is an $N>0$ such that for $n>N$, the second term $\|x^{k}-x^{k,n}\|<\epsilon$. Combining these two estimates, we conclude that $\|x^{*}-x^{k,n}\|<M\|x^{*}-x^{k-1}\|^2+\epsilon$, completing the proof.
\end{proof}
It is interesting to note that the ``error" in approximating the full Hessian by the nonlocal one has a clear effect in the convergence of the nonlocal Newton iterations. The term $M\|x^{*}-x^{k-1}\|^2$ in the proof of theorem \ref{thm:newton2} corresponds to the ``usual" quadratic convergence of the standard Newton iterations. However, the additional $\epsilon$ term corresponds to the error made in the approximation of the true Hessian by the nonlocal one. Indeed, $\epsilon$ can be made smaller by choosing a larger $n$ in the nonlocal Hessian. However, we are not guaranteed a clean quadratic convergence in the nonlocal case, indeed, the rate of convergence to $x^{*}$ is always limited by $\epsilon$, i.e., the effect of nonlocal interactions in computing the nonlocal Hessian.

\section{Concluding Remarks}\label{sec:conc}While we have developed basic first and second order nonlocal optimization methods, much remains to be done, both theoretically and numerically. For instance, one can transfer nonlocal optimization to regularized objective functions, constrained optimization and quasi-Newton methods. A bottleneck on the computational side is the quadrature used for integration. Sparse grid methods going back to Smolyak (\cite{smolyak}) can be employed for higher dimensional problems. Nonlocal gradients can also be incorporated into deep learning models during backpropagation and would be particularly useful for data that exhibits singular behavior. Our analysis assumes the parameter space $\Theta$ is Euclidean, and extending it to the case of $\Theta$ having nonzero curvature is a non-trivial open problem. Nevertheless, the results presented here can serve as a first step towards such expansions.

\bibliography{nll}

\section{Appendix}\label{sec:A_aintro}
\subsection{Nonlocal Gradients and Hessians: More Background}
We provide here some more details on nonlocal gradients and Hessians. Recall that if $u\in C^{1}(\bar{\Omega})$, or more generally, as \cite{menspec} note (see also \cite{brezis}), if $u\in W^{1,p}(\Omega)$, then the map \[y \mapsto \frac{|u(y)-u(x)|}{\|x-y\|}\rho(x-y)\in L^1(\Omega)\] for a.e. $y\in \Omega$. Thus, in these cases, the principle value integral in equation \ref{eq:nlder2} exists and coincides with the Lebesgue integral of $\frac{u(x)-u(y)}{\|x-y\|}\frac{x-y}{\|x-y\|}\rho(x-y)$ over all of $\Omega$.

Indeed, Lemma $2.1$ of \cite{menspec} states:
\begin{proposition}(Lemma $2.1$ of \cite{menspec}) Let $u\in W^{1,p}(\Omega)$. Then $\|\nabla_{\rho}u\|_{L^{p}(\Omega)}\le C\|\rho\|_{L^{1}(\Omega)}\|\nabla u\|_{L^{p}(\Omega)},$ where $C=C(p,D,\Omega)$.\end{proposition}

Recall that we write $\nabla_{\rho}^{\Omega^{*}}u(x)$ to denote the corresponding nonlocal gradient on $\Omega^{*} \subset \Omega$:
\begin{equation}\nabla_{\rho}^{\Omega^{*}}u(x):= D \int_{\Omega^{*}}\frac{u(x)-u(y)}{\|x-y\|}\frac{x-y}{\|x-y\|}\rho(x-y)dy.\label{eq:A_anlder3}\end{equation}

If the sequence of $\rho_n$ approaches a limit, albeit perhaps as distributions, we can expect the nonlocal gradients $\nabla_n$ to converge to a limiting operator as well. For example, as we remarked earlier, in case the $\rho_n$ approach $\delta_0$, we can presume that the $\nabla_n$ approach $\nabla$ as well. The mode and topology of convergence of the sequence $\nabla_n u$ to $\nabla u$ as $n\rightarrow \infty$ has been investigated by many authors. In \cite{osher}, it is formally shown that \[\frac{\partial_{n}u}{\partial_{n}x_i}(x) = \frac{\partial u}{\partial x_i}(x) + error, \]
while \cite{gunz1} show for $u\in H^1(\mathbb{R}^D)$ (or more generally for $H^1$ tensors) the convergence of the nonlocal gradient in $L^2(\Omega)$:
\[\nabla_n u \rightarrow \nabla u \text{ in $L^2(\Omega,\mathbb{R}^D)$ as }n\rightarrow \infty.\]
The localization of nonlocal gradients in multiple different topologies has been studied in \cite{menspec}, and we restate the following theorem from \cite{menspec}.
\begin{theorem}(Theorem $1.1$ of \cite{menspec}) \begin{itemize}\item Let $u\in C^{1}(\bar{\Omega})$. Then $\nabla_{n}u \rightarrow \nabla u$ locally uniformly as $n\rightarrow \infty$. If $u\in C_{c}^{1}(\Omega)$, then the convergence is uniform. \item Let $u\in W^{1,p}(\Omega)$. Then $\nabla_{n}u \rightarrow \nabla u$ in $L^p(\Omega,\mathbb{R}^D)$as $n\rightarrow \infty$.\end{itemize}\end{theorem}
We record the following theorem from \cite{spechess} regarding a similar result in the Hessian case.
\begin{theorem}(Theorems $1.4$ and $1.5$ of \cite{spechess}) \begin{itemize}\item Let $u\in C_{c}^{2}(\mathbb{R}^D)$, then $H_{n}^{4}u\rightarrow Hu$ in $L^p(\mathbb{R}^D,\mathbb{R}^{D\times D})$ as $n\rightarrow \infty$ for $1\le p \le \infty$. \item Let $u\in W^{2,p}(\Omega)$, then $H_{n}^{4}u\rightarrow Hu$ in $L^p(\mathbb{R}^D,\mathbb{R}^{D\times D})$ $n\rightarrow \infty$ for $1\le p < \infty$.\end{itemize}\end{theorem}
Note that in the $H^4_{n}$ Hessian case, the authors of \cite{spechess} define and analyze the Hessian defined on all of $\mathbb{R}^D$, due to the definition of $H^4_{n}$ involving translation terms $u(x)\mapsto u(x+h)$, yet still obtain convergence both in $L^p(\mathbb{R}^D)$ and uniform $L^{\infty}(\mathbb{R}^D)$ norms. In section \ref{sec:A_ahessian}, we will show weak convergence of the remaining Hessians $H_{n,m}^{1},H_{n}^{2}$ and $H_{n}^{3}$ to $H$ as $m,n\rightarrow \infty$. The above theorems will be one of the main results we draw upon in our analysis. In the sequel, when we use the adjective ``the" nonlocal gradient $\nabla_n$ or ``the" nonlocal Hessians $H^{k}_{n}$, with $k=1,\ldots,4$, we mean the ones defined above associated with a fixed sequence $\rho_n(\cdot)$.

\subsection{Proofs}
\subsection{Proof of Theorem \ref{thm:order1}}
In order to prove theorem \ref{thm:order1}, we will need the following lemma.
\begin{lemma}
Let $(X,\mu,\sigma(\mathcal{X}))$ be a $\sigma$-finite measure space, i.e., $X$ is a set, $\sigma(\mathcal{X})$ is a $\sigma$-algebra generated by $\mathcal{X}$, a collection of subsets  of $X$, and $\mu$ a $\sigma$-finite measure on $X$. Assume $\mu(X)<\infty$. Let $f\in L^{1}_{+}(\mu)$, i.e., $f$ is $\mu$-integrable, with $f \ge 0$ almost everywhere. Then for any $M$ with $0\le M \le \int_{X}fd\mu$, there is an $A\in \sigma(\mathcal{X})$ such that $\int_A f d\mu = M$.
\label{thm:lemma1}\end{lemma}
\begin{proof}
The proof relies on Sierpinski's theorem on non-atomic measures (see \cite{sier1},\cite{sier2}). Let $M_f = \int_X f d\mu$. Define $\lambda(A)=\int_A f d\mu$ for any $A\in \sigma(\mathcal{X})$. Since $f \in L^{1}_{+}(\mu)$, we can conclude that $\lambda:\sigma(\mathcal{X})\rightarrow [0,M_f]$ is a non-atomic measure. By Sierpinski's theorem, $\lambda$ attains its maximum and minimum. Thus, for any $M\in [0,M_f]$ there exists an $A\in \sigma(\mathcal{X})$ with $\lambda(A) = \int_A f d\mu = M.$
\end{proof}


We now prove theorem \ref{thm:order1}.
\begin{proof}
Define for for each $i=1,\ldots,D$ and any $x\in \Omega$\begin{equation}k_{u,i}^{n}(x^*,x) = D\frac{u(x^*)-u(x)}{\|x^*-x\|}\frac{x^{*}_{i}-x^{*}_{i}}{\|x^*-x\|}\rho_{n}(x^{*}-x), 
\end{equation}
and
\begin{equation}
\begin{cases}
    B_{i}^{+}&=\{x\in \mathbb{R}^{D}:x_{i}^{*} \le x_{i}\},\\
     B_{i}^{-}&=\{x\in \mathbb{R}^{D}:x_{i}^{*} > x_{i}\}.\\
\end{cases}
\end{equation}

Notice that $B_{i}^{+} \cap B_{i}^{-} =\phi$ and that $x^{*} \in B_{i}^{+}$. Now, by assumption, $u\in W^{1,p}(\Omega)$ and hence, $k_{u,i}^{n}(x^*,x) $ is integrable with respect to $x$, i.e., $k_{u,i}^{n}(x^*,\cdot) \in L^{1}(\Omega)$. Thus, for any $\epsilon>0$ there is a $\delta>0$ such that for any ball $B_{\delta}$ with volume $vol(B_{\delta})<\delta$ we have \[|\int_{B_{\delta}} k_{u,i}^{n}(x^*,x) dx|\le \int_{B_{\delta}} |k_{u,i}^{n}(x^*,x)| dx < \epsilon.\]
Pick such a ball $B_{\delta}\subset \Omega$ with $x^{*} \in B_{\delta}$. Define 
\begin{equation}
\begin{cases}
    B_{\delta, i}^{+}&=B_{\delta}\cap B_{i}^{+},\\
     B_{\delta, i}^{-}&=B_{\delta}\cap B_{i}^{-}.\\
\end{cases}
\end{equation}
Note that $x^{*} \in B_{\delta, i}^{+}$ and that $B_{\delta} = B_{\delta, i}^{+} \cup B_{\delta, i}^{-}$.

It is easy to see, based on how we have set things up, that for varying $x \in B_{\delta}$ we have that $k_{u,i}^{n}(x^*,x) \le 0 \text{ a.e.}$ on $B_{\delta, i}^{-}$ while $k_{u,i}^{n}(x^*,x) \ge 0 \text{ a.e.}$ on $B_{\delta, i}^{+}$. Indeed, since $x^{*}$ is an a.e. local minimum, $u(x^*)-u(x)\le 0$ a.e., the sign of $k_{u,i}^{n}(x^*,x)$ is determined by that of $x^{*}_{i}-x^{*}_{i}$. On $B_{\delta, i}^{+}$, we have that $x^{*}_{i}-x^{*}_{i} \le 0 $ so that $k_{u,i}^{n}(x^*,x) \ge 0$ a.e. while on $B_{\delta, i}^{-}$, we have that $x^{*}_{i}-x^{*}_{i} \ge 0 $ so that $k_{u,i}^{n}(x^*,x) \le 0$ a.e.

Therefore, we see that 
\begin{equation}
\begin{cases}
    \int_{B_{\delta, i}^{+}}k_{u,i}^{n}(x^*,x)dx&\ge 0,\\
    \int_{B_{\delta, i}^{-}}k_{u,i}^{n}(x^*,x)dx&\le 0,\\
\end{cases}
\end{equation}
and, since $B_{\delta}$ is the disjoint union $B_{\delta, i}^{+} \cup B_{\delta, i}^{-}$
\begin{equation}
\int_{B_{\delta}}k_{u,i}^{n}(x^*,x)dx = \int_{B_{\delta, i}^{+}}k_{u,i}^{n}(x^*,x)dx +\int_{B_{\delta, i}^{-}}k_{u,i}^{n}(x^*,x)dx .
\end{equation}
If $\int_{B_{\delta}}k_{u,i}^{n}(x^*,x)dx =0$, then setting $\Omega^{*}=B_{\delta}$, we are done. Otherwise, we consider two cases.
\paragraph{Case 1: $\int_{B_{\delta}}k_{u,i}^{n}(x^*,x)dx < 0$}
This is of course equivalent to \[\int_{B_{\delta, i}^{+}}k_{u,i}^{n}(x^*,x)dx < -\int_{B_{\delta, i}^{-}}k_{u,i}^{n}(x^*,x)dx.\]
We shall shortly appeal to lemma \ref{thm:lemma1}. Let $X=B_{\delta, i}^{-}$ and $\sigma(\mathcal{B}^{-})$ be the $\sigma$-algebra generated by open subsets of $B_{\delta, i}^{-}$. We have that $\sigma(\mathcal{B}^{-}) \subset \mathcal{B}$, where $\mathcal{B}$ is the Borel $\sigma$-algebra on $\mathbb{R}^D$. Let $\mu$ be the Lebesgue measure on $\mathcal{B}$ restricted to $\sigma(\mathcal{B}^{-})$, so that $\mu(B_{\delta, i}^{-})<\delta <\infty$. If we let $f(x)=-k_{u,i}^{n}(x^*,x)$, then $f\in L^{1}_{+}(\mu)$ and $f\ge 0$ almost everywhere. Finally, set \[M:=\int_{B_{\delta, i}^{+}}k_{u,i}^{n}(x^*,x)dx\] and $ M_f:=\int_{B_{\delta, i}^{-}}-k_{u,i}^{n}(x^*,x)dx>M$.
The conditions of lemma \ref{thm:lemma1} being satisfied, we are guaranteed a measurable $A \subset B_{\delta, i}^{-}$ with $\int_A f d\mu = M$. By definition, this means that $\int_A -k_{u,i}^{n}(x^*,x) dx = \int_{B_{\delta, i}^{+}}k_{u,i}^{n}(x^*,x)dx$. Thus,
\[\int_{A}k_{u,i}^{n}(x^*,x)dx+\int_{B_{\delta, i}^{+}}k_{u,i}^{n}(x^*,x)dx.\]
Setting $\Omega^{*}=A\cup B_{\delta, i}^{+}$, we have $\nabla_n^{\Omega^{*}}u(x^*)=0$

\paragraph{Case 2: $\int_{B_{\delta}}k_{u,i}^{n}(x^*,x)dx > 0$}
This equivalent to \[\int_{B_{\delta, i}^{+}}k_{u,i}^{n}(x^*,x)dx > -\int_{B_{\delta, i}^{-}}k_{u,i}^{n}(x^*,x)dx.\]
In this case, let $X=B_{\delta, i}^{+}$ and $\sigma(\mathcal{B}^{+})$ be the $\sigma$-algebra generated by open subsets of $B_{\delta, i}^{+}$ containing $x^{*}$. We have that $\sigma(\mathcal{B}^{+}) \subset \mathcal{B}$ and $\mu$ is the Lebesgue measure on $\mathcal{B}$ restricted to $\sigma(\mathcal{B}^{+})$, so that $\mu(B_{\delta, i}^{+})<\delta <\infty$. If we let $f(x)=k_{u,i}^{n}(x^*,x)$, then $f\in L^{1}_{+}(\mu)$ and $f\ge 0$ almost everywhere. Similar to the previous case, we set \[M:=-\int_{B_{\delta, i}^{-}}k_{u,i}^{n}(x^*,x)dx\] and $ M_f:=\int_{B_{\delta, i}^{+}}k_{u,i}^{n}(x^*,x)dx>M$.
As before, the conditions of lemma \ref{thm:lemma1} are satisfied, and we are guaranteed a measurable $A \subset B_{\delta, i}^{+}$ with $\int_A f d\mu = M$. By definition, this means that $\int_A k_{u,i}^{n}(x^*,x) dx = -\int_{B_{\delta, i}^{-}}k_{u,i}^{n}(x^*,x)dx$. Thus,
\[\int_{A}k_{u,i}^{n}(x^*,x)dx+\int_{B_{\delta, i}^{-}}k_{u,i}^{n}(x^*,x)dx.\]
Setting $\Omega^{*}=A\cup B_{\delta, i}^{-}$, we have $\nabla_n^{\Omega^{*}}u(x^*)=0$
\end{proof}
Note that the $\Omega^{*}$ constructed in the above proof implicitly  depends on $n$ due to the dependency of $k_{u,i}^{n}(x^*,x)$ on $n$.

\subsection{Proof of Proposition \ref{thm:taylor1}}
\begin{proof} 
\textbf{Case 1, uniform convergence: $u\in C_{c}^1(\Omega).$}
Since $r_n-r = A(x_0,x)-A_n(x_0,x)$, we show $r_n\rightarrow r$ uniformly, from which it follows that $A_n(x_0,x)\rightarrow A(x_0,x)$ uniformly. We can estimate, using Cauchy-Schwarz, \begin{equation}
\begin{split}
|r_n-r|&=|A(x_0,x)-A_n(x_0,x)|\\&=|(x-x_0)^T(\nabla u(x_0)-\nabla_n u(x_0))|\\&\le \|x-x_0\| \|\nabla u(x_0)-\nabla_n u(x_0)\|\\&\le M \|\nabla u(x_0)-\nabla_n u(x_0)\|,
\end{split}
\end{equation}
where, since $x,x_0 \in \Omega$ and $\Omega$ is bounded, we have $\|x-x_0\|<M$ for some $M>0$ independent of $x_0,x$. Since $u\in C_{c}^1(\Omega)$, we know $\nabla_n u \rightarrow \nabla u$ uniformly by Theorem 1.1 of \cite{menspec}. Thus, given $\epsilon >0$, there is an $N>0$ such that for all $n>N$, we have $\|\nabla u(x_0)-\nabla_n u(x_0)\|<\frac{\epsilon}{M}$. Putting this into the earlier estimate yields $|r_n-r|<\epsilon$ for $n>N$.

\textbf{Case 2, $L^2(\Omega\times\Omega)$ convergence: $u\in H_{0}^1(\Omega).$}We know that $C^{1}_{0}(\Omega)$ is dense in $H_{0}^1(\Omega)$ (see \cite{adams}). Hence, we select a $\tilde{u}\in C^{1}_{0}(\Omega)$ such that $\|u-\tilde{u}\|_{H^{1}_{0}(\Omega)}<\epsilon$ and we consider the standard (local) and nonlocal affine approximant of $\tilde{u}$ which we denote by $\tilde{r}$ and $\tilde{r}_n$ respectively. As before, we have \begin{equation*}
\begin{split}
|\tilde{r}_n-\tilde{r}|&\le \|x-x_0\| \|\nabla \tilde{u}(x_0)-\nabla_n \tilde{u}(x_0)\|.
\end{split}
\end{equation*}
Now, as we noted earlier, we have that $\|x-x_0\|<M$ for some $M>0$ independent of $x_0,x$ due to $\Omega$ being bounded. Thus, holding $x$ constant, we have,
 \begin{equation}
\begin{split}
\int_{\Omega}|\tilde{r}_n(\cdot,x)-\tilde{r}(\cdot,x)|^2 dx_0&\le M^2 \int_{\Omega}\|\nabla \tilde{u}(x_0)-\nabla_n \tilde{u}(x_0)\|^2 dx_0 \\&= M^2\|\nabla \tilde{u}(x_0)-\nabla_n \tilde{u}(x_0)\|^{2}_{L^2(\Omega)}.
\end{split}
\end{equation}
Assume that the measure of $\Omega$ is $m=m(\Omega)$. Then, integrating again with respect to $x$ we have  \begin{equation}
\begin{split}
\int_{\Omega}\int_{\Omega}|\tilde{r}_n(\cdot,\cdot)-\tilde{r}(\cdot,\cdot)|^2 dx_0 dx&\le (M \,m)^2 \int_{\Omega}\|\nabla \tilde{u}(x_0)-\nabla_n \tilde{u}(x_0)\|^2 dx_0 \\&= (M \,m)^2\|\nabla \tilde{u}(x_0)-\nabla_n \tilde{u}(x_0)\|^{2}_{L^2(\Omega)}.
\end{split}
\end{equation} By Theorem 1.1 of \cite{menspec}, $\|\nabla \tilde{u}(x_0)-\nabla_n \tilde{u}(x_0)\|_{L^2(\Omega)}\rightarrow 0 $ as $n\rightarrow \infty$, so we conclude that given $\epsilon>0$, we can find an $N>0$ such that \[\|\nabla \tilde{u}(x_0)-\nabla_n \tilde{u}(x_0)\|_{L^2(\Omega)}<\frac{\epsilon}{M\,m}\] when $n>N$. Thus, we have for $n>N$, that \begin{equation}\|\tilde{r}-\tilde{r}_n\|_{L^2(\Omega \times \Omega)}<\epsilon.\label{eq:A_aeqaa}\end{equation} We now write \begin{equation*}\|r-r_n\|_{L^2(\Omega \times \Omega)}\le \|r-\tilde{r}\|_{L^2(\Omega \times \Omega)}+\|\tilde{r}-\tilde{r}_n\|_{L^2(\Omega \times \Omega)}+\|\tilde{r}_n-r_n\|_{L^2(\Omega \times \Omega)},\end{equation*}
and estimate each of the three terms in the above equation. Now,
\begin{equation*}
\begin{split}
\|r-\tilde{r}\|^{2}_{L^2(\Omega\times\Omega)} &\le 2 m^2 (\|u(x)-\tilde{u}(x)\|_{L^2(\Omega)}^{2}\\& + \|u(x_0)-\tilde{u}(x_0)\|^{2}_{L^2(\Omega)}+M^2\|\nabla u(x_0) - \nabla \tilde{u}(x_0)\|^{2}_{L^2(\Omega)}).
\end{split}
\end{equation*}
By the density assumption, we have that $\|u-\tilde{u}\|_{L^2(\Omega)}<\epsilon$ and $\|\nabla u-\nabla \tilde{u}\|_{L^2(\Omega)}<\epsilon$, and hence, we conclude that \begin{equation}\|r-\tilde{r}\|_{L^2(\Omega\times\Omega)}\le m \sqrt{4+2M^2} \epsilon.\label{eq:A_aeqbb}\end{equation}
Similarly, we have
\begin{equation*}
\begin{split}
\|r_n-\tilde{r}_n\|^{2}_{L^2(\Omega\times\Omega)} &\le 2 m^2 (\|u(x)-\tilde{u}(x)\|_{L^2(\Omega)}^{2}\\& + \|u(x_0)-\tilde{u}(x_0)\|^{2}_{L^2(\Omega)}+M^2\|\nabla_n u(x_0) - \nabla_n \tilde{u}(x_0)\|^{2}_{L^2(\Omega)}).
\end{split}
\end{equation*}
We also have 
\begin{equation}
\begin{split}
\|\nabla_n u(x_0) - \nabla_n \tilde{u}(x_0)\|_{L^2(\Omega)} &\le \|\nabla_n u(x_0) - \nabla u(x_0)\|_{L^2(\Omega)}\\&+\|\nabla u(x_0) - \nabla \tilde{u}(x_0)\|_{L^2(\Omega)}+\|\nabla \tilde{u}(x_0) - \nabla_n \tilde{u}(x_0)\|_{L^2(\Omega)}.
\end{split}\label{eq:A_aeq11}
\end{equation}
By theorem 1.1 of \cite{menspec}, there exists an $N_1>0$ such that each of the three terms in equation \ref{eq:A_aeq11} can be bounded by $\frac{\epsilon}{3}$, for all $n>N_1$, which yields \begin{equation}\|\nabla_n u(x_0) - \nabla_n \tilde{u}(x_0)\|_{L^2(\Omega)} < \epsilon,\end{equation}
for all $n>N_1$, which in turn implies that \begin{equation}\|r_n-\tilde{r}_{n}\|_{L^2(\Omega\times\Omega)}\le m \sqrt{4+2M^2} \epsilon.\label{eq:A_aeqcc}\end{equation}
Putting equations \ref{eq:A_aeqaa}, \ref{eq:A_aeqbb} and \ref{eq:A_aeqcc} together, we have, for $n>\text{max}(N,N_1)$, that \[\|r-r_n\|_{L^{2}(\Omega\times\Omega)}<(1+2m\sqrt{4+2M^2})\epsilon,\] and since $\epsilon$ was arbitrary, we conclude that $\|r_n(\cdot,\cdot)- r(\cdot,\cdot)\|_{L^2(\Omega\times \Omega)}\rightarrow 0$ as $n\rightarrow \infty$.
\end{proof}

\subsection{Proof of Proposition \ref{thm:bdd}}
\begin{proof}
Let $u\in Lip(\Omega,M)$. We know from Lemma 2.2 of \cite{menspec} that $\|\nabla_n u\|_{L^{\infty}(\Omega)}\le DM\|\rho_n\|_{L^1(\Omega)}$ and since $\|\rho_n\|_{L^1(\Omega)}\le\|\rho_n\|_{L^1(\mathbb{R}^D)}=1$, we conclude that \[\|\nabla_n u\|_{L^{\infty}(\Omega)}\le DM.\] Now, \[\|x^{k+1,n}\|\le \|x^{k,n}\|+\alpha^{k}\|\nabla_n u^{k,n}\|\le \|x^{k,n}\|+DM\alpha^{k}.\] Iterating this observation, we get \[\|x^{k+1,n}\|\le \|x^{0}\|+DM(\sum_{m=1}^{k}\alpha^m).\] By assumption, we have $\sum_{m=1}^{K}\alpha^{m}<1$ for any $K>0$. Thus, $\|x^{k+1,n}\|\le \|x^{0}\| +DM$, and the sequence of iterates $\{x^{k+1,n}\}$ is bounded.
\end{proof}

\subsection{Proof of Theorem \ref{thm:suboptimalgd}}
\begin{proof}
The proof is by induction on $k$.
\paragraph{Base case:}Let $k=0$. Now, $x^{0} = x^{0,n}$, so \[x^{1}-x^{1,n} = \alpha^0(\nabla_n u^0-\nabla u^0).\]
By \cite{menspec}, we know that $\nabla_n u\rightarrow \nabla u$ as $n\rightarrow \infty$ uniformly for all $x\in V$. Thus, given $\epsilon > 0 $ there is an $N>0$ so that for all $n>N$, we have $\|\nabla_n u^0- \nabla u^0\|<\frac{\epsilon}{\alpha^0}$. Thus, $\|x^{1}-x^{1,n}\|<\epsilon$ for $n>N$.
\paragraph{Induction Step:}Having completed the base case, we assume the result to be true for $k=K$, and we prove the estimate for $k=K+1$. Again, \[x^{K+1}-x^{K+1,n} = x^{K}-x^{K,n} + \alpha^K(\nabla_n u^{K,n}-\nabla u^K).\] Thus, \[\|x^{K+1}-x^{K+1,n}\|\le\underbrace{\|x^{K}-x^{K,n}\|}_{(1)} + \alpha^K\underbrace{\|(\nabla_n u^{K,n}-\nabla u^K)\|}_{(2)}.\] By the induction hypothesis, there is an $N_0>0$ such that term $(1)$ is bounded by $\frac{\epsilon}{3}$ for $n>N_0$.

For the second term, we have \[\|\nabla_n u^{K,n}-\nabla u^K\|\le\|\nabla_n u^{K,n}-\nabla u^{K,n}\|+\|\nabla u^{K,n}-\nabla u^K\|.\] First, by uniform convergence of  $\nabla_n u\rightarrow \nabla u$ that there is an $N_1$ such that for $n>N_1$, we have \[\|\nabla_n u^{K,n}-\nabla u^{K,n}\|<\frac{\epsilon}{3 \alpha^{K}}.\] Second, by the mean-value theorem applied to $\nabla u^{K,n}-\nabla u^{K}$, we have that \[\|\nabla u^{K,n}-\nabla u^{K}\|\le C\|x^{K,n}-x^{K}\|,\] where $C=\|Hu\|_{L^{\infty}(V)}<\infty$. Indeed, since $u\in C^2(\Omega)$ the Hessian is continuous on the compact set $V$ and hence is uniformly bounded. By the induction hypothesis, there is an $N_2$ such that $\|x^{K,n}-x^{K}\|<\frac{\epsilon}{3C\alpha^K}$ for all $n>N_2$. Choosing $N=\max({N_0,N_1,N_2})$, we have that for $n>N=N(\alpha^0,\ldots,\alpha^{K})$, $\|x^{K+1}-x^{K+1,n}\|<\epsilon$, and the induction step is complete.
\end{proof}

\subsection{Proof of Theorem \ref{thm:optimalgd}}
In the proof of the theorem, we shall need the notion of $\Gamma-$convergence in the context of metric spaces (see \cite{gamma} for more details).
\begin{definition}
Let $X$ be a metric space (or, more generally, a first countable topological space) and $F_n:X\rightarrow \overline{\mathbb{R}}$ be a sequence of extended real valued functions on $X$. Then the sequence $F_{n}$ is said to $\Gamma$-converge to the $\Gamma$-limit $F:X\to \overline {\mathbb {R}}$ (written $F_n \xrightarrow{\text{$\Gamma$}} F$) if the following two conditions hold:
\begin{itemize}\item For every sequence $x_n \in X$ such that $x_n\rightarrow x \in X$ as $n\rightarrow \infty$, we have $F(x) \le \lim\inf_{n\rightarrow \infty} F_n(x_n),$ \item For all $x\in X$, there is a sequence of $x_n \in X$ with $x_n\rightarrow x$ (called a $\Gamma$-realizing sequence) such that $F(x)\ge \lim\sup_{n\rightarrow \infty}F_n(x_n).$\end{itemize}
\end{definition}
The notion of $\Gamma$-convergence is an essential ingredient for the study of convergence of minimizers. Indeed, if $F_n \xrightarrow{\text{$\Gamma$}} F$, and $x_{n}$ is a minimizer of $F_{n}$, then it is true that every limit point of the sequence $x_{n}$ is a minimizer of $F$. Moreover, if $F_n \rightarrow F$ uniformly, then $F_n \xrightarrow{\text{$\Gamma$}} F^{lsc}$, where $F^{lsc}$ is the \emph{lower semi-continuous} envelope of $F$, i.e., the largest lower semi-continuous function bounded above by $F$. In particular, if $F_n$ and $F$ happen to be continuous, then $F^{lsc}=F$, and therefore, if $F_n \rightarrow F$ uniformly then $F_n \xrightarrow{\text{$\Gamma$}} F^{lsc}=F$ (see \cite{gamma}).
\begin{proof}
The proof is by induction on $k$.
\paragraph{Base case: }
We consider $k=0$. Recall that \[\alpha^{0,n}=\argminA_{\alpha\in I}\,u(x^{0,n}-\alpha\nabla_n u^{0,n})=\argminA_{\alpha\in I}\,u(x^{0}-\alpha\nabla_n u^{0}),\]
since $x^0=x^{0,n}$.
Now let \[\phi_{n}(\alpha)=x^{0}-\alpha\nabla_n u^{0},\] and \[\phi(\alpha)=x^{0}-\alpha\nabla u^{0}.\]We first notice that $\phi_n, \phi$ are continuous and $\phi_n\rightarrow \phi$ uniformly as $n\rightarrow \infty$. Indeed, \[|\phi_n(\alpha)- \phi(\alpha)| = \alpha|\nabla u^0-\nabla_n u^0|,\] and, by the uniform $\alpha$-independent convergence $\nabla_n u^0\rightarrow \nabla u^0$ (see \cite{menspec}), we conclude that \[\|\phi_n(\alpha)- \phi(\alpha)\|_{L^{\infty}(\Omega)} = \alpha\|\nabla u^0-\nabla_n u^0\|_{L^{\infty}(\Omega)}\rightarrow 0\] as $n\rightarrow \infty$ for any $\alpha \in I$. Now, since $I$ is compact and $u,\phi_n,\phi$ are continuous, the composition $u\circ \phi_n$ converges uniformly to $u \circ \phi$.

By our remarks about $\Gamma$-convergence earlier, we know that uniform convergence of a sequence of continuous functions implies $\Gamma$-convergence, which in turn implies convergence of limits of minimizers. Thus, limit points of the sequence of minimizers $\{\alpha^{0,n}\}$ of $u\circ \phi_n$ are minimizers of $u\circ \phi$, i.e., $\alpha^0$. In other words, there is a subsequence $\{\alpha^{0,n_m(0)}\}$ of $\{\alpha^{0,n}\}$ such that \[\alpha^{0,n_m(0)}\rightarrow \alpha^{0}\] as $n_m(0)\rightarrow \infty$.
\begin{equation}
\begin{split}
    \|x^{1,n_m(0)}-x^{1}\| &= \|\alpha^{0}\nabla u^0 - \alpha^{0,n_m(0)}\nabla_{n_m(0)}u^{0,n_m(0)}\|\\
&=\|\alpha^{0}\nabla u^0 - \alpha^{0,n_m(0)}\nabla_{n_m(0)}u^{0}\|\\
&= |\alpha^0-\alpha^{n_m(0)}|\|\nabla u^0 - \nabla_{n_m(0)}u^{0}\|,\\
\end{split}
\end{equation}
because $x^{0}=x^{0,n_m(0)}$ so that,
\begin{equation}
\begin{split}
    \nabla_{n_m(0)} u^{0,n_m(0)} &= \nabla_{n_m(0)} u(x^{0,n_m(0)})\|\\
&=\nabla_{n_m(0)} u(x^{0})\\
&=\nabla_{n_m(0)} u^{0}.\\
\end{split}
\end{equation}
Since $\alpha^{0,n_m(0)}\rightarrow \alpha^{0}$ as $n_m(0)\rightarrow \infty$, for any $\epsilon>0$, there is an $N_1>0$ such that \[|\alpha^0-\alpha^{n_m(0)}|<\sqrt{\epsilon}\] for $n_m(0)>N_1$. Likewise, there is an $N_2>0$ such that \[\|\nabla u^0 - \nabla_{n_m(0)}u^{0,n_m(0)}\|=\|\nabla u^0 - \nabla_{n_m(0)}u^{0}\|<\sqrt{\epsilon}\] for $n_m(0)>N_2$ since $\nabla_{n} u \rightarrow \nabla u$ as $n\rightarrow \infty$. Choosing $N=\max\{N_1,N_2\}$, we see $\|x^{1}-x^{1,n_m(0)}\|<\epsilon$. The base case is proved.

\paragraph{Induction Step: }We assume the result is true for $k=K-1$. We prove it for $k=K$. Similar to the base case, we let \[\phi_{K,n}(\alpha)=x^{K,n}-\alpha\nabla_n u^{K,n}\] and \[\phi_{K}(\alpha)=x^{K}-\alpha\nabla u^{K}.\] We now show that by the induction hypothesis, there is a subsequence $\phi_{K,n_m(K)}$ of $\phi_{K,n}$ such that $\phi_{K,n_m(K)}\rightarrow \phi_{K}$ uniformly. Indeed, \[\phi_{K,n}(\alpha)-\phi_{K}(\alpha) = x^{K,n}-x^{K}+\alpha(\nabla u^{K}-\nabla_{n}u^{K,n}).\]  Thus, \[\|\phi_{K,n}(\alpha)-\phi_{K}(\alpha)\|\le \|x^{K,n}-x^{K}\|+\alpha\|\nabla u^{K}-\nabla_{n}u^{K,n}\|,\] so that  \[\|\phi_{K,n}(\alpha)-\phi_{K}(\alpha)\|\le \underbrace{\|x^{K,n}-x^{K}\|}_{(1)}+\alpha(\underbrace{\|\nabla u^{K}-\nabla u^{K,n}\|}_{(2)}+\underbrace{\|\nabla u^{K,n}-\nabla_{n}u^{K,n}\|}_{(3)}).\] By the induction hypothesis, there is a subsequence $x^{K,n_m(K)}$ of $x^{K,n}$ such that $x^{K,n_m(K)} \rightarrow x^{K}$ as $n_m(K)\rightarrow \infty$. Thus, given an $\epsilon>0$, there is an $N_1>0$ such that for all $n_m(K)>N_1$, \[\|x^{K,n_m(k)}-x^{K}\|<\frac{\epsilon}{3}.\] Next, by the uniform convergence of $\nabla_n u \rightarrow \nabla u$, we have that there exists an $N_2>0$ such that \[\|\nabla u^{K,n_m(K)}-\nabla_{n_m(K)}u^{K,n_m(K)}\| < \frac{\epsilon}{3}\] for all $n_m(K)>N_2$. As in the suboptimal stepsize case, we now apply the mean value theorem to $\nabla u^{K}-\nabla u^{K,n_m(K)}$. We have that \[\|\nabla u^{K}-\nabla u^{K,n_m(K)}\|\le C\|x^{K}-x^{K,n_m(K)}\|< C\frac{\epsilon}{3}\] where $C=\|Hu\|_{L^{\infty}(V)}<\infty$. Indeed, since $u\in C^2(\Omega)$, the Hessian is continuous on the compact set $V$ containing the iterates $x^{k},x^{k,n}$ and hence is uniformly bounded. Putting all this together, and noting that $\alpha \in I = [0,A]$ we see that \[\|\phi_{K,n_m(K)}(\alpha)-\phi_{K}(\alpha)\|< \frac{\epsilon}{3}(1+\alpha(C+1))\le \frac{\epsilon}{3}(1+A(C+1))\] for $n_m(K)>\max{\{N_1,N_2\}}$. Thus, $\phi_{K,n_m(K)}(\alpha)\rightarrow \phi_{K}(\alpha)$ uniformly on the compact interval $I$. Therefore, the composition $u\circ\phi_{K,n_m(K)}$ converges uniformly to $u\circ \phi_{K}$ since $u$ is continuous. This implies that $u\circ\phi_{K,n_m(K)}$ $\Gamma$-converges to $u\circ \phi_{K}$, so that every limit point of the collection of minimizers $\{\alpha^{K,n_m(K)}\}$ of $\phi_{K,n_m(K)}$ is a (and since we have assumed uniqueness of the linesearch minimizer, we can say \emph{the}) minimizer $\alpha^{K}$ of $\phi_{K}$. There is therefore a further subsequence of $\{\alpha^{K,n_m(K)}\}$, which we will denote by $\{\alpha^{K,n_{m_{1}}(K)}\}$, such that $\alpha^{K,n_{m_{1}}(K)}\rightarrow \alpha^{K}$. This proves the first half of the result. We still need to show that there exists a subsequence $x^{K+1,n_{m_{2}}(K)}$ of $x^{K+1,n}$ such that $\|x^{K+1,n_{m_{2}}(K)}-x^{K+1}\|\rightarrow 0$ as $n_{m_{2}}(K)\rightarrow \infty$.

To this end, consider \begin{equation}
\begin{split}
    \|x^{K+1,n}-x^{K+1}\| &= \|x^{K,n}-x^{K}+\alpha^{K}\nabla u^{K}-\alpha^{K,n}\nabla_{n}u^{K,n}\|\\
&\le \|x^{K,n}-x^{K}\|+\|\alpha^{K}\nabla u^{K}-\alpha^{K,n}\nabla_{n}u^{K,n}\|.\\
\end{split}
\end{equation}
We write \begin{equation}
\begin{split}
    \|\alpha^{K}\nabla u^{K}-\alpha^{K,n}\nabla_{n}u^{K,n}\| &= \|\alpha^{K}\nabla u^{K}-\alpha^{K}\nabla_{n}u^{K,n}+\alpha^{K}\nabla_{n}u^{K,n}-\alpha^{K,n}\nabla_{n}u^{K,n}\|\\
&\le \alpha^{K}\|\nabla u^{K}-\nabla_{n}u^{K,n}\|+|\alpha^{K}-\alpha^{K,n}|\|\nabla_{n}u^{K,n}\|\\
&\le \alpha^{K}\|\nabla u^{K}-\nabla_{n}u^{K,n}\|+|\alpha^{K}-\alpha^{K,n}|DM.
\end{split}
\end{equation}The last inequality follows since $u\in C^2(\Omega) \cap Lip(\Omega,M)$ and by Lemma 2.2 of \cite{menspec}, we have $\|\nabla_{n}u^{K,n}\|\le \|\nabla_{n} u\|_{L^{\infty}(\Omega)}\le DM\|\rho\|_{L^1(\Omega)}\le DM.$

Now, \[\|\nabla u^{K}-\nabla_{n}u^{K,n}\|\le \|\nabla u^{K}-\nabla u^{K,n}\|+\|\nabla u^{K,n}-\nabla_{n}u^{K,n}\|.\] By the mean value theorem, the first term \[\|\nabla u^{K}-\nabla u^{K,n}\|\le C\|x^{K}-x^{K,n}\|\] with $C=\|Hu\|_{L^{\infty}(V)}<\infty$ due to $u\in C^2(\Omega)$. The second term \[\|\nabla u^{K,n}-\nabla_{n}u^{K,n}\|\rightarrow 0\] as $n\rightarrow \infty$ by uniform convergence of $\nabla_n u$ to $\nabla u$ on $V$. Thus, putting all the above together, we see that
\[
\|x^{K+1,n}-x^{K+1}\|\le (1+\alpha^{K}C)\|x^{K,n}-x^{K}\|+\alpha^{K}\|\nabla u^{K,n}-\nabla_n u^{K,n}\|+|\alpha^{K,n}-\alpha^{K}|DM.
\]

Now, by our induction hypothesis, there is a subsequence $x^{K,n_m(K)}$ of $x^{K,n}$ such that $x^{K,n_m(K)} \rightarrow x^{K}$ as $n_m(K)\rightarrow \infty$. Thus, given an $\epsilon>0$, there is an $N_1>0$ such that for all $n_m(K)>N_1$, \[\|x^{K,n_m(k)}-x^{K}\|<\frac{\epsilon}{3(1+\alpha^{K}C)}.\] Next, by what was proved earlier, there is a subsequence $\alpha^{K,n_{m_{1}}(K)}$ of $\alpha^{K,n_m(K)}$ and an $N_2>0$ such that for $n_{m_1}(K)>N_3$, \[|\alpha^{K,n_{m_{1}}(K)}-\alpha^{K}|<\frac{\epsilon}{3DM}.\] Finally, there exists an $N_3$ such that  \[\|\nabla u^{K,n_m(K)}-\nabla_{n_m(K)}u^{K,n_m(K)}\|<\frac{\epsilon}{3\alpha^K}.\] Define a new subsequence $\{n_{m_{2}}(K)\}$ to be the diagonal of $\{n_m(K)\}$ and $\{n_{m_{1}}(K)\}$. Then, for $n_{m_{2}}(K)>\max\{N_1,N_2,N_3\}$, we have 
\begin{equation}
\begin{split}
    \|x^{K+1,n_{m_2}(K)}-x^{K+1}\|&\le (1+\alpha^{K}C)\|x^{K,n_{m_2}(K)}-x^{K}\|\\
&+\alpha^{K}\|\nabla u^{K,n_{m_2}(K)}-\nabla_n u^{K,n_{m_2}(K)}\|+|\alpha^{K,n_{m_2}(K)}-\alpha^{K}|DM \\&< \epsilon.\\
\end{split}
\end{equation}
The subsequence $x^{K+1,n_{m_2}(K)}$ is the required subsequence of $x^{K+1,n}$, and the induction is complete.
\end{proof}

\subsection{Proof of Lemma \ref{thm:sgd1}}
\begin{proof}
By convexity of $u$, we have that
\begin{equation}\label{eq:A_asgd1}
\begin{split}
u(y)-u(x)&\ge (y-x)^T\nabla u(x), \text{ or,}\\
u(x)-u(y)&\le (x-y)^T\nabla u(x).\\
\end{split}
\end{equation}On a compact neighborhood $K$ of $x$, $\nabla_n u \rightarrow \nabla u$ uniformly by theorem 1.1 of \cite{menspec}. Thus, $\nabla_n u(x) \rightarrow \nabla u(x)$ as $n\rightarrow \infty$. For any $\lambda \in \mathbb{R}^D$, this implies that we have $\lambda^T\nabla_n u(x)\rightarrow \lambda^T\nabla u(x)$. Choosing $\lambda = (y-x)$, we have $(y-x)^T\nabla_n u(x)\rightarrow (y-x)^T\nabla u(x)$. Thus, given $\epsilon>0$, there is an $N>0$ such that for all $n>N$, we have $|(y-x)^T\nabla_n u(x) - (y-x)^T\nabla u(x)|<\epsilon$, so that $-\epsilon + (y-x)^T\nabla_n u(x) < (y-x)^T\nabla u(x) < \epsilon + (y-x)^T\nabla_n u(x)$.  Thus, $u(y)-u(x)\ge (y-x)^T\nabla u(x) > (y-x)^T\nabla_n u(x) -\epsilon.$ 
\end{proof}
\subsection{Proof of Theorem \ref{thm:sgdsgd1}}
As a motivating discussion of the nonlocal stochastic gradient descent method, we consider a simple example. If $f:[0,1]\rightarrow \mathbb{R}$ is a differentiable function, we can consider, for any $x,y \in[0,1]$, and $x\neq y$, the difference quotient \[k_f(x,y) = \frac{f(x)-f(y)}{x-y},\]which, by the mean value theorem, equals $f^{\prime}(\xi)$ for some $\xi\in[x,y]$. If we now think of $X,Y$ as $[0,1]-$valued random variables, then we have that the difference quotient $k_{f}(X,Y)$ is also a random variable. Thus, it is reasonable to expect that if the the conditional distribution of $Y$ given $X$ is centered sufficiently closely around $X$, then \[\mathbb{E}[k_{f}(X,Y)|X]\approx f^{\prime}(X).\] While this approximation depends on the exact nature of $p_{Y|X}(y|x)$, it nevertheless motivates our interpretation of $k_{f}(X,Y)$ as being a nonlocal subgradient.
 
We recall a relevant result from \cite{shai}.
\begin{corollary}(Corollary 14.2 of \cite{shai}) Let $u$ be a convex, $M-$Lipschitz function and let $x^{*}\in \arg\min_{x:\|x\|\le B}u(x).$ If we run the GD algorithm on $u$ for $K$ steps with $\alpha=\sqrt{\frac{B^2}{M^2 K}}$ then the output vector $\bar{x}$ satisfies \[u(\bar{x}) -u(x^{*})\le \frac{BM}{\sqrt{K}}.\] Furthermore, for every $\epsilon>0$, to achieve $u(\bar{x})-u(x^*)\le \epsilon,$ it suffices to run the GD algorithm for a number of iterations that satisfies \[K\ge \frac{B^2M^2}{\epsilon^2}.\]
\end{corollary}
We now prove \ref{thm:sgdsgd1}, which we view as a nonlocal version of corollary 14.2 of \cite{shai} mentioned above.
\begin{proof}
Consider $u(\bar{x}) -u(x^{*}) = u(\frac{1}{K}\sum_{k=1}^{K}x^{k}) -u(x^{*})$\begin{equation}
\begin{split}
u(\bar{x}) -u(x^{*})&=u(\frac{1}{K}\sum_{k=1}^{K}x^{k}) -u(x^{*})\\
&\le \frac{1}{K}\sum_{k=1}^{K}u(x^{k}) -u(x^{*})\\
&= \frac{1}{K}(\sum_{k=1}^{K}u(x^{k}) -u(x^{*})).\\
\end{split}
\end{equation}
By the choice of $\epsilon-$subgradient $g^k \in \partial_{\epsilon}u(x^k)$ of $u$, we have for every $k$, that 
\begin{equation}u(x^k)-u(x^{*})\le  (x^k-x^*)^Tg^k+\epsilon,\end{equation}
which upon summing and dividing by $K$ yields \begin{equation}u(\bar{x}) -u(x^{*})\le \frac{1}{K}\sum_{k=1}^{K}(x^k-x^*)^Tg^k+\epsilon.\end{equation}
We appeal now to lemma 14.1 of \cite{shai} to conclude that $\frac{1}{K}\sum_{k=1}^{K}(x^k-x^*)^Tg^k\le \frac{BM}{\sqrt{K}}$, which results in the final estimate \begin{equation}u(\bar{x}) -u(x^{*})\le \frac{BM}{\sqrt{K}}+\epsilon.\end{equation}
Now, given an $\hat{\epsilon}>\epsilon$, it is clear that picking $K\ge \frac{B^2M^2}{(\hat{\epsilon}-\epsilon)^2}$ results in $u(\bar{x}) -u(x^{*})\le \hat{\epsilon}$.
\end{proof}
\subsection{Proof of Theorem \ref{thm:sgdsgd2}}

\begin{proof}
The proof follows that of theorem 14.8 of \cite{shai}. Let $g^{1:k}$ denote the sequence of $\epsilon-$subgradients $g^1,\ldots,g^k$, and $\mathbb{E}_{g^{1:k}}[\cdot]$ indicates the expected value of the bracketed quantity over the draws $g^1,\ldots,g^k$. Now,  $u(\bar{x}) -u(x^{*})\le \frac{1}{K}(\sum_{k=1}^{K}u(x^{k}) -u(x^{*}))$, and therefore, $\mathbb{E}_{g^{1:K}}[u(\bar{x}) -u(x^{*})]\le \mathbb{E}_{g^{1:K}}[\frac{1}{K}(\sum_{k=1}^{K}u(x^{k}) -u(x^{*}))].$ By lemma 14.1 of \cite{shai}, we have $\mathbb{E}_{g^{1:K}}[\frac{1}{K}\sum_{k=1}^{K}(g^k)^T(x^k-x^{*})\le \frac{BM}{\sqrt{K}}.$

We show now that \begin{equation}\label{eq:A_asgdsgd}\mathbb{E}_{g^{1:K}}[\frac{1}{K}(\sum_{k=1}^{K}u(x^{k}) -u(x^{*}))]\le \mathbb{E}_{g^{1:K}}[\frac{1}{K}\sum_{k=1}^{K}(g^k)^T(x^k-x^{*})]+\epsilon,\end{equation}
which would establish the result.

Now,\begin{equation}
\begin{split}
\mathbb{E}_{g^{1:K}}[\frac{1}{K}\sum_{k=1}^{K}(g^k)^T(x^k-x^{*})] &=\frac{1}{K}\sum_{k=1}^{K}\mathbb{E}_{g^{1:K}}[(g^k)^T(x^k-x^{*})]\\
&= \frac{1}{K}\sum_{k=1}^{K}\mathbb{E}_{g^{1:k}}[(g^k)^T(x^k-x^{*})],\\
\end{split}
\end{equation}
where the first equality is by linearity of expectation, and the second equality is due to the fact that the quantity $(g^k)^T(x^k-x^{*})$ depends on iterates upto and including $k$ only, and not beyond $k$. Next, by the law of total expectation, we have \begin{equation}
\begin{split}
\mathbb{E}_{g^{1:k}}[(g^k)^T(x^k-x^{*})] &=\mathbb{E}_{g^{1:k-1}}[\mathbb{E}_{g^{1:k}}[(g^k)^T(x^k-x^{*})|g^{1:k-1}]]. \\
\end{split}
\end{equation}
Conditioned on $g^{1:k-1}$, the quantity $x^{k}$ is deterministic, and hence
\begin{equation}
\begin{split}
\mathbb{E}_{g^{1:k-1}}[\mathbb{E}_{g^{1:k}}[(g^k)^T(x^k-x^{*})|g^{1:k-1}]]&=\mathbb{E}_{g^{1:k-1}}[(\mathbb{E}_{g^{k}}[g^k|g^{1:k-1}])^T(x^k-x^{*})]. \\
\end{split}
\end{equation}
We now make the crucial observation that $\mathbb{E}_{g^{k}}[g^k|g^{1:k-1}] \in \partial_{\epsilon} u(x^k)$. Indeed, we have noted that $x^{k}$ is deterministic given $g^{1:k-1}$, so that $\mathbb{E}_{g^{k}}[g^k|g^{1:k-1}] = \mathbb{E}_{g^{k}}[\mathbb{E}_{g^{k}}[g^k|g^{1:k-1}]|x^k]$. However, since $x^k$ depends on $g^{1:k-1}$, the Doob-Dynkin lemma (see \cite{olav}) implies that the sigma algebra $\sigma(x^k)$ generated by $x^k$ is a subalgebra of $\sigma(g^{1:k-1})$, the sigma algebra generated by $g^{1:k-1}$. Therefore, by the tower property of conditional expectation, we have that $\mathbb{E}_{g^{k}}[\mathbb{E}_{g^{k}}[g^k|g^{1:k-1}]|x^k] = \mathbb{E}_{g^{k}}[g^k|x^k]$. This implies that $\mathbb{E}_{g^{k}}[g^k|g^{1:k-1}] \in \partial_{\epsilon} u(x^k)$ since by construction we know that $\mathbb{E}_{g^{k}}[g^k|x^k]\in \partial_{\epsilon} u(x^k)$. Thus, by the $\epsilon-$subgradient property of the $g^k$, we have that $\mathbb{E}_{g^{1:k-1}}[(\mathbb{E}_{g^{k}}[g^k|g^{1:k-1}])^T(x^k-x^{*})] \ge \mathbb{E}_{g^{1:k-1}}[u(x^k)-u(x^{*})]-\epsilon.$ 

Thus, we have shown that $\mathbb{E}_{g^{1:K}}[(g^k)^T(x^k-x^{*})]\ge \mathbb{E}_{g^{1:k-1}}[u(x^k)-u(x^{*})]-\epsilon=\mathbb{E}_{g^{1:K}}[u(x^k)-u(x^{*})]-\epsilon.$ Summing over $k=1,\ldots,K$ and dividing by $K$ we obtain equation \ref{eq:A_asgdsgd}, and thereby we have established the result.
\end{proof}
\subsection{Proof of Theorem \ref{thm:hess1a}}
We now provide the proof of theorem \ref{thm:hess1a}. The proof is a direct consequence of Theorem 1.1 of \cite{menspec}.
\begin{proof}
\textbf{Convergence of $(H_{n,m}^{1})_{i,j} \xrightarrow{m\rightarrow\infty, \, n\,(\text{fixed}) >N }(H_{n}^{3})_{i,j}$: }
Let $K\subset \Omega$ be compact. Assume that an $N$ has been chosen large enough so as to ensure that $u_{x_i}^{n} \in C^{1}(\bar{\Omega})$ for all $n>N$. Fix such an $n>N$. Then, applying Theorem 1.1 of \cite{menspec}, we obtain that $(u_{x_i}^{n})_{x_j}^{m}\rightarrow (u_{x_i}^{n})_{x_j}$  as $m\rightarrow \infty$, i.e., $(H_{n,m}^{1})_{i,j} \xrightarrow{m\rightarrow\infty, \, n\,(\text{fixed}) >N }(H_{n}^{3})_{i,j}$ locally on $K$. If $u\in C^2_{c}(\Omega),$ the uniform convergence of $(H_{n,m}^{1})_{i,j} \xrightarrow{m\rightarrow\infty, \, n\,(\text{fixed}) >N }(H_{n}^{3})_{i,j}$ follows from the analogous result of Theorem 1.1 of \cite{menspec}.

\textbf{Convergence of $(H_{n}^{2})_{i,j} \xrightarrow{n\rightarrow\infty}(H)_{i,j}$: }
The proof of this case is exactly like that of the previous case obtained by replacing $u_{x_i}^{n}$ with $u_{x_i}$. With the analogous set-up as before, we have $u_{x_i} \in C^1(\bar{\Omega})$. Thus, applying Theorem 1.1 of \cite{menspec}, we conclude that $(u_{x_i})_{x_j}^{n}\rightarrow (u_{x_i})_{x_j}$ locally uniformly. If $u\in C^2_{c}(\Omega),$ the uniform convergence of $(H_{n}^{2})_{i,j} \rightarrow (H)_{i,j}$ follows from the analogous result of Theorem 1.1 of \cite{menspec}.

\end{proof}

\subsection{Proof of Theorem \ref{thm:hess2a}}
We now provide the proof of theorem \ref{thm:hess2a}.
\begin{proof}
Let $v\in H^1_{0}(\Omega)$. We show that $((u^{n}_{x_i})_{x_j}-(u_{x_i})_{x_j},v)_{L^2}\rightarrow 0$ as $n\rightarrow \infty$.

Now, integrating by parts we obtain,
\[
((u^{n}_{x_i})_{x_j}-(u_{x_i})_{x_j},v)_{L^2} = -((u^{n}_{x_i})-(u_{x_i}),v_j)_{L^2},\]
where the boundary term is zero due to $v$ coming from $H^1_{0}$ and having zero trace. We can now estimate:
\begin{equation}
\begin{split}
|((u^{n}_{x_i})_{x_j}-(u_{x_i})_{x_j},v)_{L^2}|&=|((u^{n}_{x_i})-(u_{x_i}),v_j)_{L^2}|\\
&\le\|u^{n}_{x_i}-u_{x_i}\|_{L^2}\|v_j\|_{L^2}\\
&\le\|u^{n}_{x_i}-u_{x_i}\|_{L^2}\|v\|_{H^{1}_{0}},\\
\end{split}
\end{equation}
where the first inequality above is Cauchy-Schwarz, while the second follows from the definition of the $H^{1}_{0}$ norm. Now, since $u\in C^2(\bar{\Omega})$, in particular, we have that $u\in H^2(\Omega)$ and therefore $u_{x_i} \in H^1(\Omega)$. By \cite{menspec}, Theorem 1.1, we have that $\|u^{n}_{x_i}-u_{x_i}\|_{L^2}\rightarrow 0$ as $n\rightarrow \infty$ for $u_{x_i} \in H^1(\Omega)$. Therefore, given $\epsilon > 0$, there is an $N>0$ such that for all $n>N$, we have $\|u^{n}_{x_i}-u_{x_i}\|_{L^2} < \frac{\epsilon}{\|v\|}$. Thus, $|((u^{n}_{x_i})_{x_j}-(u_{x_i})_{x_j},v)_{L^2}| < \epsilon$ for $n>N$. Since $\epsilon$ is arbitrary, we have proved $((u^{n}_{x_i})_{x_j}-(u_{x_i})_{x_j},v)_{L^2}\rightarrow 0$ as $n\rightarrow \infty$.

The remainder of Theorem \ref{thm:hess2a} is a reorganization of the statements of Theorem \ref{thm:hess1a} along with what was just proved.
\end{proof}

\subsection{Proof of Theorem \ref{thm:hess3}}
The proof follows that of Theorem 1.1 of \cite{menspec}. We record here the following Lemma from \cite{menspec} that we will need in our proof as well.
\begin{lemma}(Lemma $3.1$ of \cite{menspec}) Let \[c_{n}^{i}(x) := \int_{\Omega}\frac{(x_i-y_i)^2}{\|x-y\|^2}\rho_{n}(x-y)dy.\] Then $Dc_{n}^{i}(x)\rightarrow 1$ pointwise as $n\rightarrow \infty$, and the convergence is uniform on compact sets.
\label{thm:menLemma}
\end{lemma}
We now provide the proof of theorem \ref{thm:hess3}.
\begin{proof}
For $x\in K$, with $K\subset \Omega$, $K$ being compact, we show that $(u_{x_i}^{n})_{x_j}^{m}\rightarrow (u_{x_i}^{n})_{x_j}$ uniformly as $m,n\rightarrow \infty$. Let \begin{equation}J_{i,j}^{n,m}=|D\int_{\Omega} \frac{u_{x_i}^{n}(x)-u_{x_i}^{n}(y)}{\|x-y\|}\frac{x_{j}-y_{j}}{\|x-y\|}\rho_{m}(x-y)dy-Dc_{j}^{m}(x)(u_{x_i})_{x_j}|,\end{equation}
where we have used the fact that $(u_{x_i})_{x_j}$ exists due to $u\in C^2(\bar{\Omega})$.  Now, since $Dc_{j}^{m}(x)\rightarrow 1$ uniformly  as $m\rightarrow \infty$ by lemma \ref{thm:menLemma}, we have that $Dc_{j}^{m}(x)(u_{x_i})_{x_j} \rightarrow (u_{x_i})_{x_j}, \, m\rightarrow \infty$ uniformly.

We show that there exists an $N>0$ such that for $n>N$ fixed, we have $J_{i,j}^{n,m}\rightarrow 0$ uniformly for $x\in K$ as $m\rightarrow \infty$, which implies that $(u_{x_i}^{n})_{x_j}^{m}\rightarrow (u_{x_i})_{x_j}$ uniformly as $m\rightarrow \infty$, i.e., $(H_{n,m}^{1})_{i,j} \xrightarrow{m\rightarrow\infty, \, n\,(\text{fixed}) >N }(H)_{i,j}$.

Now, we have by definition \[c_{j}^{m}(x) := \int_{\Omega}\frac{(x_j-y_j)^2}{\|x-y\|^2}\rho_{m}(x-y)dy,\]
so that \begin{equation}
\begin{split}
J_{i,j}^{n,m}&=|D(\int_{\Omega} \frac{u_{x_i}^{n}(x)-u_{x_i}^{n}(y)}{\|x-y\|}\frac{x_{j}-y_{j}}{\|x-y\|}\rho_{m}(x-y)dy-\int_{\Omega}\frac{(x_j-y_j)^2}{\|x-y\|^2}(u_{x_i})_{x_j}\rho_{m}(x-y)dy)| \\
 & = |D(\int_{\Omega} \frac{u_{x_i}^{n}(x)-u_{x_i}^{n}(y)-(x_j-y_j)(u_{x_i})_{x_j}}{\|x-y\|}\frac{x_{j}-y_{j}}{\|x-y\|}\rho_{m}(x-y)dy)|. \\
\end{split}\label{eq:A_ahess3eq1}
\end{equation}
Now, by Theorem 1.1 of \cite{menspec}, we have that for $x\in K$, $u_{x_i}^{n}\rightarrow u_{x_i}$ uniformly. Thus, given $\epsilon>0$, there exists an $N_\epsilon>0$ independent of $x\in K$ such that for all $n>N_\epsilon$, \[\|u_{x_i}^{n}-u_{x_i}\|_{L^\infty(\Omega)}<\epsilon,\]
or, for any $x,y\in K$, $|u_{x_i}^{n}(x)-u_{x_i}(x)|<\epsilon$ and $|u_{x_i}(y)-u_{x_i}^{n}(y)|<\epsilon$, which upon combining together, result in \begin{equation}-2\epsilon+ u_{x_i}(x)-u_{x_i}(y)<u_{x_i}^{n}(x)-u_{x_i}^{n}(y)<2\epsilon+ u_{x_i}(x)-u_{x_i}(y).\label{eq:A_ahess3eq2}\end{equation}Putting \ref{eq:A_ahess3eq2} into \ref{eq:A_ahess3eq1}, we obtain 

\begin{equation}
\begin{split}
J_{i,j}^{n,m}&\le|D(\int_{\Omega} \frac{2\epsilon + u_{x_i}(x)-u_{x_i}(y)-(x_j-y_j)(u_{x_i})_{x_j}}{\|x-y\|}\frac{x_{j}-y_{j}}{\|x-y\|}\rho_{m}(x-y)dy)|\\&\le\underbrace{|D(\int_{\Omega} \frac{2\epsilon(x_{j}-y_{j})}{\|x-y\|^2}\rho_{m}(x-y)dy)|}_{(1)}\\&+\underbrace{|D(\int_{\Omega} \frac{u_{x_i}(x)-u_{x_i}(y)-(x_j-y_j)(u_{x_i})_{x_j}}{\|x-y\|}\frac{x_{j}-y_{j}}{\|x-y\|}\rho_{m}(x-y)dy)|}_{(2)}.
\end{split}\label{eq:A_ahess3eq3}
\end{equation}
Term $(1)$ in equation \ref{eq:A_ahess3eq3} can be estimated as:
\begin{equation} |D(\int_{\Omega} \frac{2\epsilon(x_{j}-y_{j})}{\|x-y\|^2}\rho_{m}(x-y)dy)|\le 2D\epsilon \int_{\Omega}\frac{\rho_m(x-y)}{\|x-y\|}dy\le2DM\epsilon,\label{eq:A_aestim}\end{equation}
since, by the assumption on $\rho_m$, we have $\int_{\Omega}\frac{\rho_m(x-y)}{\|x-y\|}dy\le\int_{\mathbb{R}^D}\frac{\rho_m(x-y)}{\|x-y\|}dy<M(m)$.
We now focus on term $(2)$ in equation \ref{eq:A_ahess3eq3}.

Since by assumption $(u_{x_i})\in C^{1}(\bar{\Omega})$, i.e., $(u_{x_i})_{x_j} \in C^{0}(\bar{\Omega})$, we can find, given an $\epsilon_2 >0$, a $\delta>0$ small enough so that $\|x-y\|<\delta$ implies that $\|(u_{x_i}^{n})_{x_j}(x)-(u_{x_i}^{n})_{x_j}(y)\|<\epsilon_2$.

Assume an $\epsilon_2$ is given and the corresponding $\delta$ has been found. Thus, we can write 
\begin{equation}
\begin{split}
&|D(\int_{\Omega} \frac{u_{x_i}(x)-u_{x_i}(y)-(x_j-y_j)(u_{x_i})_{x_j}}{\|x-y\|}\frac{x_{j}-y_{j}}{\|x-y\|}\rho_{m}(x-y)dy)|\\&=|D(\int_{B_\delta(x)} \frac{u_{x_i}(x)-u_{x_i}(y)-(x_j-y_j)(u_{x_i})_{x_j}}{\|x-y\|}\frac{x_{j}-y_{j}}{\|x-y\|}\rho_{m}(x-y)dy\\&+\int_{\|x-y\|>\delta} \frac{u_{x_i}(x)-u_{x_i}(y)-(x_j-y_j)(u_{x_i})_{x_j}}{\|x-y\|}\frac{x_{j}-y_{j}}{\|x-y\|}\rho_{m}(x-y))| \\
 & \le |D(\int_{B_\delta(x)} \frac{u_{x_i}(x)-u_{x_i}(y)-(x_j-y_j)(u_{x_i})_{x_j}}{\|x-y\|}\frac{x_{j}-y_{j}}{\|x-y\|}\rho_{m}(x-y))dy|\\&+\frac{2D}{\delta}(\|u_{x_i}\|_{L^{\infty}}+\|\nabla u_{x_i}\|_{L^{\infty}})\int_{\|x-y\|>\delta} \rho_{m}(x-y)dy. \\
\end{split}\label{eq:A_ahess3eq4}
\end{equation}
By the mean value theorem, $\forall y \in B_{\delta}(x),\,\exists\eta(x,y)\in B_{\delta}(x)$ such that \[u_{x_i}(x)-u_{x_i}(y) = (u_{x_i})_{x_j}(\eta(x,y))\cdot(x_j-y_j).\]
Putting in this equality into equation \ref{eq:A_ahess3eq4}, we obtain:
\begin{equation}
\begin{split}
&|D(\int_{B_\delta(x)} \frac{u_{x_i}(x)-u_{x_i}(y)-(x_j-y_j)(u_{x_i})_{x_j}}{\|x-y\|}\frac{x_{j}-y_{j}}{\|x-y\|}\rho_{m}(x-y))dy| \\& \le |D(\int_{B_\delta(x)} \frac{(u_{x_i})_{x_j}(\eta(x,y))-(u_{x_i})_{x_j}}{\|x-y\|}\frac{(x_{j}-y_{j})^2}{\|x-y\|}\rho_{m}(x-y))dy|\\& \le |D\epsilon_2\int_{B_\delta(x)}\rho_{m}(x-y)dy|.\\
\end{split}
\end{equation}
By the assumption on the radial densities $\rho_{m}$, we have that $\int_{B_\delta(x)}\rho_{m}(x-y)dy<1$ and $\int_{\|x-y\|>\delta} \rho_{m}(x-y)dy \rightarrow 0,\, m\rightarrow 0$. Putting equations \ref{eq:A_ahess3eq3}, \ref{eq:A_aestim}, and \ref{eq:A_ahess3eq4} together, we obtain $\lim_{m\rightarrow \infty} \sup_{x\in K} J_{i,j}^{n,m} \le 2D \epsilon \lim_{m\rightarrow \infty}M(m)+D\epsilon_2$ for $n>N_\epsilon$. As $\epsilon,\epsilon_2$ were arbitrary, we obtain the result that the $L^{\infty}$ error of $\|(H_{n,m}^{1})_{i,j} - H_{i,j}\|_{L^{\infty}(K)}$ as $m\rightarrow \infty$ is bounded by $2D\epsilon\lim_{m\rightarrow \infty}M(m)$.

If $u\in C_{c}^{2}(\Omega)$, and $M(m)$ is independent of $m$, we can conclude uniform convergence on $\Omega$ using an argument similar to that of Theorem 1.1 of \cite{menspec} adapted to the current situation. Indeed, if $u\in C^{2}_{c}(\Omega)$ then $\text{supp}(u)$ is a proper closed subset of $\Omega$. Thus, the distance $d:=dist(\partial \Omega,\text{supp}(u))$ between $\partial \Omega$ and $\text{supp}(u)$ is strictly positive. If \[\Omega_{\frac{d}{2}}:=\{x \in \Omega:dist(\{x\},\partial \Omega)>\frac{d}{2}\},\]
then $\sup_{x\in \Omega_{\frac{d}{2}}}J^{n,m}_{i,j}\rightarrow 0, \, m\rightarrow \infty$. For points $x\in \Omega \cap \Omega_{\frac{d}{2}}^{c}$, we have \[J^{n,m}_{i,j}\le \frac{2D}{d}(\|u\|_{L^{\infty}}+\|\nabla u\|_{L^{\infty}})\int_{\|z\|>\frac{d}{2}}\rho_{m}(z)dz \rightarrow 0, \, m\rightarrow \infty.\]
\end{proof}
\subsection{Proof of Theorem \ref{thm:newton1}}
\begin{proof}
By theorem 1.1 of \cite{menspec}, we have $\nabla_n u \rightarrow \nabla u$ locally uniformly, and extending $u$ smoothly by zero outside of $\Omega$, we have uniform convergence on all of $\mathbb{R}^D$. By theorem 1.4 of \cite{spechess}, we have $H_n u \rightarrow H u$ uniformly on $\mathbb{R}^{D\times D}$. By assumption on $(Hu(x))^{-1}$ being uniformly continuous on $U$, it follows that $(H_n u(x))^{-1}$ converges uniformly to $(Hu(x))^{-1}$ on $U$. Therefore, $(H_{n}u)^{-1}\nabla_n u$ converges uniformly to $(H u)^{-1}\nabla u$ on $U$. We now prove the result by induction. 
\paragraph{Base case $k=0$: } We have:
\begin{equation}
\begin{cases}
   x^1&=x^{0}-\beta^0 (Hu^0)^{-1} \nabla u^0,\\
     x^{1,n}&=x^{0,n}-\beta^0 (H_n u^{0,n})^{-1} \nabla_n u^{0,n},\\
\end{cases}
\end{equation}
so that \[\|x^1-x^{1,n}\|=\beta^0\|(H_{n}u^{0,n})^{-1}\nabla_n u^{0,n}-(H u^{0})^{-1}\nabla u^{0}\|.\]
Thus, for $x^{0,n}=x^{0} \in U$, we have, given an $\epsilon>0$, an $N>0$ such that $\|(H_{n}u^{0,n})^{-1}\nabla_n u^{0,n}-(H u^{0})^{-1}\nabla u^{0}\|<\frac{\epsilon}{\beta^0}$ for all $n>N$ We can therefore conclude that $\|x^1-x^{1,n}\|<\epsilon$ for $n>N$ and the base case is complete.
\paragraph{Induction case $k=K$: } Assume that the result is true for $k=K-1$, and we proceed to prove the induction step $k=K$. Let $\epsilon>0$ be given. Consider \begin{equation}x^{K+1}-x^{K+1,n}=(x^{K}-x^{K,n})+\beta^K((H_n u^{K,n})^{-1}\nabla_n u^{K,n}-(H u^{K})^{-1} \nabla u^{K}).\end{equation} We have \begin{equation}\label{eq:A_anewton1}\|x^{K+1}-x^{K+1,n}\|\le \underbrace{\|x^{K}-x^{K,n}\|}_{(1)}+\beta^K\underbrace{\|(H_n u^{K,n})^{-1}\nabla_n u^{K,n}-(H u^{K})^{-1} \nabla u^{K}\|}_{(2)}.\end{equation}
By the induction hypothesis, there is an $N_1>0$ such that for $n>N_1$ we have $\|x^{K}-x^{K,n}\|<\frac{\epsilon}{2}.$ We now consider term $(2)$ of equation \ref{eq:A_anewton1}. We write:
\begin{equation}\label{eq:A_anewton2}
\begin{split}
\|(H_n u^{K,n})^{-1}\nabla_n u^{K,n}-(H u^{K})^{-1} \nabla u^{K}\|&=\underbrace{\|(H_n u^{K,n})^{-1}\nabla_n u^{K,n}-(H u^{K,n})^{-1} \nabla u^{K,n}\|}_{(1)}\\
&+\underbrace{\|(H u^{K,n})^{-1} \nabla u^{K,n}-(H u^{K})^{-1} \nabla u^{K}\|}_{(2)}.\\
\end{split}
\end{equation}
Since $(H_n u)^{-1}\nabla_n u\rightarrow (H u)^{-1} \nabla u$ uniformly on $U$, we have that there exists an $N_2>0$ such that $\|(H_n u^{K,n})^{-1}\nabla_n u^{K,n}-(H u^{K,n})^{-1} \nabla u^{K,n}\|<\frac{\epsilon}{4\beta^K}$ for $n>N_2$. Next, we use the assumption that $f=(H u)^{-1} \nabla u \in C^{1}_{c}(U)$. This implies, by the mean value theorem that $\|(H u^{K,n})^{-1} \nabla u^{K,n}-(H u^{K})^{-1} \nabla u^{K}\|\le C\|x^{K,n}-x^{K}\|$ where $C=\|\nabla f\|_{L^{\infty}}<\infty.$ Therefore, there exists an $N_3>0$ such that for $n>N_3$, we have $\|x^{K,n}-x^{K}\|<\frac{\epsilon}{4C\beta^K}$. Putting these estimates together into equation \ref{eq:A_anewton2} and \ref{eq:A_anewton1}, we get $\|x^{K+1}-x^{K+1,n}\|<\epsilon$ for $n>N=\max\{N_1,N_2,N_3\}$. The induction step is proved.
\end{proof}
\vskip 0.2in

\end{document}